\documentclass[a4paper,12pt]{amsart}
\usepackage{amssymb}
\RequirePackage{calrsfs}
\DeclareSymbolFont{rsfscript}{OMS}{rsfs}{m}{b}
\DeclareSymbolFontAlphabet{\mathrsfs}{rsfscript}
\renewcommand{\mathcal}{\mathrsfs}

\def\bfit{\bfseries\itshape}

\input xypic
\xyoption{all}
\xyoption{arc}

\newtheorem{theo}{Theorem}[section]

\newtheorem{prop}[theo]{Proposition}

\newtheorem{lem}[theo]{Lemma}

\newtheorem{coro}[theo]{Corollary}

\def\remark#1{{\refstepcounter{theo}\label{#1}\noindent\sc Remark
\arabic{section}.\arabic{theo} - }}
\def\example#1{{\refstepcounter{theo}\label{#1}\noindent\sc Example 
\arabic{section}.\arabic{theo} - }}

\def\equat{\refstepcounter{theo}$$~}
\def\endequat{\leqno{\boldsymbol{(\arabic{section}.\arabic{theo})}}~$$}

\newtheorem{theoap}{Theorem}

\newtheorem{propap}[theoap]{Proposition}

\newtheorem{coroap}[theoap]{Corollary}

\def\exampleap#1{{\refstepcounter{theoap}\label{#1}\noindent\sc Example 
\Alph{section}.\arabic{theoap} - }}

\def\equatap{\refstepcounter{theoap}$$~}
\def\endequatap{\leqno{
\boldsymbol{({\mathbf{\Alph{section}}}.\arabic{theoap})}}~$$}

\def\AG{{\mathfrak A}}

    \def\NM{{\mathbb{N}}}

    \def\QM{{\mathbb{Q}}}
    \def\RM{{\mathbb{R}}}
\def\SG{{\mathfrak S}}

    \def\ZM{{\mathbb{Z}}}

  \def\ab{{\mathbf a}}  
    
    \def\CC{{\mathcal{C}}}
    \def\DC{{\mathcal{D}}}
    \def\EC{{\mathcal{E}}}
    \def\FC{{\mathcal{F}}}
    
    \def\HC{{\mathcal{H}}}
    \def\IC{{\mathcal{I}}}

    \def\LC{{\mathcal{L}}}

    \def\PC{{\mathcal{P}}}
    
    \def\RC{{\mathcal{R}}}

    \def\UC{{\mathcal{U}}}

    \def\XC{{\mathcal{X}}}

\def\Brm{{\mathrm{B}}}

    \def\HCB{{\boldsymbol{\mathcal{H}}}}

 \def\HCt{{\widetilde{\mathcal{H}}}}
    
\def\Jti{{\widetilde{J}}}

  \def\tti{{\tilde{t}}}

\def\Wti{{\widetilde{W}}}

\def\Iba{{\bar{I}}}

\def\Sba{{\bar{S}}}          \def\sba{{\bar{s}}}
          \def\tba{{\bar{t}}}

\def\Cov{{\overline{C}}}

          \def\hov{{\overline{h}}}

\def\g{\gamma}
\def\G{\Gamma}

\def\ph{\varphi}
\def\l{\lambda}
\def\L{\Lambda}
\def\o{\omega}

\def\r{\rho}
\def\s{\sigma}

\def\th{\theta}

\def\t{\tau}

\def\phb{{\boldsymbol{\varphi}}}        \def\pht{{\tilde{\varphi}}}

\def\phba{{\bar{\varphi}}}

\DeclareMathOperator{\Hom}{{\mathrm{Hom}}}

\DeclareMathOperator{\Id}{{\mathrm{Id}}}

\DeclareMathOperator{\Ind}{{\mathrm{Ind}}}

\DeclareMathOperator{\Ker}{{\mathrm{Ker}}}

\DeclareMathOperator{\sgn}{{\mathrm{sgn}}}

\DeclareMathOperator{\Pos}{{\mathrm{Pos}}}

\DeclareMathOperator{\trans}{{\mathrm{trans}}}

\def\to{\rightarrow}
\def\longto{\longrightarrow}
\def\injto{\hookrightarrow}

\def\fonction#1#2#3#4#5{\begin{array}{rccc}
{#1} : & {#2} & \longto & {#3} \\
& {#4} & \longmapsto & {#5} 
\end{array}}

\def\fonctio#1#2#3#4{\begin{array}{ccc}
{#1} & \longto & {#2} \\
{#3} & \longmapsto & {#4} 
\end{array}}

\def\vide{\varnothing}

\def\DS{\displaystyle}
\def\SS{\scriptstyle}

\def\finl{~$\SS \square$}

\def\infspe{\hspace{0.1em}\mathop{\preccurlyeq}\nolimits\hspace{0.1em}}

\def\lexp#1#2{\kern\scriptspace\vphantom{#2}^{#1}\kern-\scriptspace#2}
\def\le{\mathop{\leqslant}\nolimits}
\def\ge{\mathop{\geqslant}\nolimits}

\mathchardef\inferieur="321E
\mathchardef\superieur="321F

\def\eqna{\begin{eqnarray*}}
\def\endeqna{\end{eqnarray*}}

\def\tors{tores maximaux }

\def\tors{{\mathrm{tors}}}

\def\itemth#1{\item[${\mathrm{(#1)}}$]}

\catcode`\@=11
\long\def\@car#1#2\@nil{#1}
\long\def\@first#1#2{#1}
\long\def\@second#1#2{#2}
\long\def\ifempty#1{\expandafter\ifx\@car#1@\@nil @\@empty
  \expandafter\@first\else\expandafter\@second\fi}
\catcode`\@=12

\def\positif{\PC\! os}
\def\chambres{{\mathrm{Cham}}}
\def\facets{{\mathrm{Fac}}}
\def\relations{\RC\! el}
\DeclareMathOperator{\can}{{\mathrm{can}}}
\DeclareMathOperator{\Cell}{{\mathrm{Cell}}}
\DeclareMathOperator{\maps}{{\mathrm{Maps}}}
\DeclareMathOperator{\weight}{{\mathrm{Weights}}}
\DeclareMathOperator{\codim}{{\mathrm{codim}}}

\def\ellb{\boldsymbol{\ell}}
\def\ellt{\tilde{\ell}}

\def\aug{{\mathrm{aug}}}
\def\espace{\vphantom{\DS{\frac{A}{A}}}}
\def\dotcup{\hskip1mm\dot{\cup}\hskip1mm}

\def\gauche{\stackrel{L,\ph}{\longleftarrow}}
\def\droite{\stackrel{R,\ph}{\longleftarrow}}
\def\bilatere{\stackrel{LR,\ph}{\longleftarrow}}
\def\pre#1{\leqslant_{#1}}

\DeclareMathOperator{\SGN}{{\mathrm{SGN}}}

\begin{document}

\baselineskip=15pt

\title{Semicontinuity properties of \\ Kazhdan-Lusztig cells}

\author{C\'edric Bonnaf\'e}
\address{\noindent 
Labo. de Math. de Besan\c{c}on (CNRS: UMR 6623), 
Universit\'e de Franche-Comt\'e, 16 Route de Gray, 25030 Besan\c{c}on
Cedex, France} 

\makeatletter
\email{cedric.bonnafe@univ-fcomte.fr}
\urladdr{http://www-math.univ-fcomte.fr/pp\_Annu/CBONNAFE/}

\thanks{The author is partly supported by the ANR (Project 
No JC07-192339)}


\date{\today}

\begin{abstract} 
Computations in small Coxeter groups or dihedral groups suggest 
that the partition into Kazhdan-Lusztig cells with unequal 
parameters should obey to some semicontinuity phenomenon 
(as the parameters vary). The aim of this paper is to 
provide a rigorous theoretical background for supporting 
this intuition that will allow to state several precise conjectures. 
\end{abstract}

\maketitle

\pagestyle{myheadings}

\markboth{\sc C. Bonnaf\'e}{\sc Semicontinuity properties of 
Kazhdan-Lusztig cells}


Let $(W,S)$ be a Coxeter group and assume, for simplification 
in this introduction, that $S$ is finite and that $S= S_1 \dotcup S_2$, 
where $S_1$ and $S_2$ are two non-empty subsets of $S$ such that no element 
of $S_1$ is conjugate to an element of $S_2$. Let 
$\ell : W \to \NM=\{0,1,2,\dots\}$ be the length function 
and $\ell_i : W \to \NM$ the $S_i$-length function (if $w \in W$, 
$\ell_i(w)$ is the number of occurrences of elements of $S_i$ in a 
reduced decomposition of $w$). Let us fix two non-zero natural 
numbers $a$ and $b$ and let $L_{a,b} : W \to \ZM$, 
$w \mapsto a \ell_1(w) + b \ell_2(w)$. Then $L_{a,b}$ 
is a {\it weight function} (in the sense of Lusztig) 
so it is possible to define a partition of $W$ into left 
Kazhdan-Lusztig cells. It is clear 
that this partition depends only on $b/a$: we shall denote it 
by $\LC_{b/a}(W)$. Explicit computations suggest the following 
conjecture:

\bigskip

\begin{quotation}
\noindent{\bf Conjecture 0.} 
{\it There exist rational numbers $0 < r_1 < \cdots < r_m$ 
(depending only on $(W,S)$) such that (setting $r_0=0$ and $r_{m+1}=+\infty$), 
if $\th$ and $\th'$ are two positive rational numbers, then:
\begin{itemize}
\itemth{a} If $0 \le i \le m$ and $r_i < \th,\th' < r_{i+1}$, 
then $\LC_\th(W)=\LC_{\th'}(W)$. 

\itemth{b} If $1 \le i \le m$ and $r_{i-1} < \th < r_i < \th' < r_{i+1}$, 
then $\LC_{r_i}(W)$ is the finest partition of $W$ which is less fine than  $\LC_\th(W)$ and less fine than $\LC_{\th'}(W)$. 
\end{itemize}}
\end{quotation}

\bigskip

\noindent{\sc Remarks - } (1) 
One can obviously state similar conjectures for the partitions 
into right or two-sided cells.

\smallskip

(2) In the case where $W$ is finite, the existence of rational 
numbers $0 < r_1 < \cdots < r_m$ satisfying (a) is easy 
(see Proposition \ref{finitude}). However, 
even in this case, the statement (b) is still a conjecture.\finl

\bigskip

The aim of this paper is to provide a theoretical 
background that allows us to generalize this conjecture 
in the following two directions:
\begin{quotation}
\noindent $\bullet$ It is of course possible to work with a partition 
of $S$ into more than two subsets (this does not occur if $W$ is 
finite and irreducible).

\noindent $\bullet$ One can also be interested in weight functions 
with values in any totally ordered abelian group, and whose values 
on simple reflections are not necessarily positive.
\end{quotation}
These two generalisations lead us to define equivalence 
classes of weight functions (in the easy case detailed 
in this introduction, this has been done by working with 
the ratio $b/a$ instead of the pair $(a,b)$) and to use 
a topology on the set of such equivalence classes. This is 
done by using the notion of {\it positive subsets} of 
an abelian group (as defined in \cite[\S 1]{bonnafe positif}): we shall 
review some of the results of \cite{bonnafe positif} in an 
Appendix. 

This paper is organized as follows. After two sections devoted 
to recollections of well-known facts about Hecke algebras and 
Kazhdan-Lusztig theory with unequal parameters, we shall state 
our main conjecture in the third one: it is only concerned 
with the partition into cells and will be given into 
different forms (see Conjectures A, A' and A''). In the fourth section, 
we shall also state a conjecture saying, roughly speaking, that 
Conjecture A is compatible with the construction of {\it cell representations} 
(see Conjecture B). In the last section, we 
illustrate these conjectures by giving a detailed account of the following examples:

\bigskip

\noindent{\sc Examples - } 
(1) {\it Dihedral groups:} 
If $|S|=2$ and if $|S_1|=|S_2|=1$, then the Conjecture 0 
holds by taking $m=1$ and $r_1=1$ (see \cite[\S 8.8]{lusztig}). 

\medskip

(2) {\it Type $F_4$:} If $(W,S)$ is of type $F_4$ and if $|S_1|=|S_2|=2$, 
then the Conjecture 0 holds by taking $m=3$ and $r_1=1/2$, $r_2=1$ et $r_3=2$ 
(see \cite[Corollaire 4.8]{geck f4}). 

\medskip

(3) {\it Type $B_n$:} Assume that $(W,S)$ is of type 
$B_n$ (with $n \ge 2$) and that $|S_1|=n-1$ and $|S_2|=1$. 
In \cite[Conjectures A and B]{bgil}, the Conjecture 0 is made more 
precise: it should be sufficient to take $m=n-1$ and $r_i=i$. 
This has been checked for $n \le 6$. 

\medskip

(4) {\it Type $\tilde{G}_2$:} Assume that $(W,S)$ is the affine Weyl group of type 
$\tilde{G}_2$ and that $S_1$ and $S_2$ are chosen in such a way that $|S_1|=2$ and $|S_2|=1$. 
Then Guilhot \cite{guilhot rang 2} has shown that the Conjecture 0 holds by 
taking $m=3$ and $r_1/1$, $r_2=3/2$ and $r_3=2$.

\medskip

(5) {\it Type $\tilde{B}_2$:} Guilhot \cite{guilhot rang 2} 
has also shown that Conjecture 0 holds if $(W,S)$ 
is an affine Weyl group of type $\tilde{B}_2$. In this case, there are several possibilities 
for the choice of the partition $S=S_1 \dotcup S_2$, but Guilhot has proved that 
Conjecture 0 holds for all possible choices.\finl

\bigskip


\bigskip

\noindent{\bf Acknowledgements.} Part of this work was done while the 
author stayed at the MSRI during the winter 2008. The author wishes to 
thank the Institute for its hospitality and the organizers of the two 
programs for their invitation. The author also thanks 
M. Geck, L. Iancu and J. Guilhot for many fruitful discussions.

\tableofcontents

\section{Hecke algebras} 

\medskip

\subsection{Preliminaries}
Let $(W,S)$ be a Coxeter system ($S$ and $W$ can be infinite). 
If $s$, $t \in S$, we shall write $s \sim t$ if they are conjugate 
in $W$. Let $\Sba=S/\!\!\sim$. 
If $s \in S$, we denote by $\sba$ its class in $\Sba$. Let 
$\ZM[\Sba]$ the free $\ZM$-module with basis $\Sba$. 

Let $\ell : W \to \NM$ be the length function associated 
with $S$. If $\o \in \Sba$ and if $w \in W$, we denote by 
$\ell_\o(w)$ the number of occurrences of elements of $\o$ 
in a reduced expression of $w$ (it is well-known that it does
not depend on the choice of the reduced expression). We set 
$$\fonction{\ellb}{W}{\ZM[\Sba]}{w}{\DS{\sum_{\o \in \Sba}} 
\ell_\o(w) \o.}$$

If $\G$ is an abelian group, a map $\ph : W \to \G$ 
is called a {\it weight function} if $\ph(ww')=\ph(w)+\ph(w')$ 
for all $w$, $w' \in W$ such that $\ell(ww')=\ell(w)+\ell(w')$ 
(see \cite[\S 3.1]{lusztig}). 
The maps $\ell_\o$ (viewed as functions with values in $\ZM$) 
and $\ellb$ are weight functions. In fact, the map $\ellb$ is universal 
in the following sense:

\bigskip

\begin{lem}\label{factorisation poids}
Let $\ph : W \to \G$ be a weight function. Then there exists 
a unique morphism of groups $\phba : \ZM[\Sba] \to \G$ 
such that $\ph = \phba \circ \ellb$. 
\end{lem}

\bigskip

\begin{proof}
Clear.
\end{proof}

\bigskip

Let $\maps(\Sba,\G)$ be the set of maps $\Sba \to \G$ and let 
$\weight(W,\G)$ be the set of weight functions $W \to \G$. 
The lemma \ref{factorisation poids} shows that there are canonical 
bijections 
$$\weight(W,\G) \stackrel{\sim}{\longleftrightarrow} \maps(\Sba,\G) 
\stackrel{\sim}{\longleftrightarrow} \Hom(\ZM[\Sba],\G).$$
We shall identify these three sets all along this paper. 
More precisely, we shall work with maps $\ph : \Sba \to \G$ 
that we shall see, according to our needs, as morphisms of groups 
$\ZM[\Sba] \to \G$ or as weight functions $W \to \G$. In particular, 
we can talk as well about $\ph(w)$ (for $w \in W$) as about 
$\ph(\l)$ (for $\l \in \ZM[\Sba]$): we hope that it will not lead 
to some confusion. For instance, $\Ker \ph$ is a subgroup 
of $\ZM[\Sba]$ (and not of $W$!).

\bigskip

\subsection{Hecke algebras with unequal parameters} 
Let us first fix the notation that will be used throughout this paper.

\medskip

\begin{quotation}
\noindent{\bf Notation.} {\it Let $\G$ be an abelian group and 
let $\ph : \Sba \to \G$ be a map.} 
\end{quotation}

\medskip

\noindent We shall use an exponential notation for the group algebra of $\G$: 
$\ZM[\G]=\DS{\mathop{\oplus}_{\g \in \G}} \ZM e^\g$, where $e^\g\cdot e^{\g'}=e^{\g+\g'}$ for all $\g$, $\g' \in \G$. We then denote by 
$\HC(W,S,\ph)$ the {\it Hecke algebra with parameter $\ph$} that is, 
the free $\ZM[\G]$-module with basis $(T_w)_{w \in W}$ endowed 
with the unique $\ZM[\G]$-bilinear associative multiplication 
completely determined by the following rules:
$$\begin{cases}
T_w T_{w'} = T_{ww'} & \text{if $\ell(ww')=\ell(w)+\ell(w')$,}\\
(T_s - e^{\ph(s)})(T_s+e^{-\ph(s)})=0 & \text{if $s \in S$.}
\end{cases}$$
If necessary, we denote by $T_w^\ph$ the element $T_w$ to emphasize 
that it lives in the Hecke algebra with parameter $\ph$. 

This algebra is endowed with several involutions. We shall use 
only the following: if $\g \in \G$ and $w\in W$, we set  $\overline{e^\g}=e^{-\g}$ and $\overline{T}_w=T_{w^{-1}}^{-1}$ 
(note that $T_w$ is invertible). This extends by $\ZM$-linearity 
to a $\ZM[\G]$-antilinear map $\HC(W,S,\ph) \to \HC(W,S,\ph)$, 
$h \mapsto \hov$ which is an involutive antilinear automorphism 
of ring. 

\medskip

The previous construction is functorial. If $\r : \G \to \G'$ is a 
morphism of abelian groups, then $\r$ induces a map 
$$\r_* : \HC(W,S,\ph) \longto \HC(W,S,\r \circ \ph)$$
defined as the unique $\ZM[\G]$-linear map sending $T_w^\ph$ on 
$T_w^{\r \circ \ph}$ (here, $\HC(W,S,\r \circ \ph)$ 
is viewed as a $\ZM[\G]$-algebra through the morphism 
of rings $\ZM[\G] \to \ZM[\G']$ induced by $\r$). It is then 
easily checked that 
\equat\label{rho morphisme}
\text{\it $\r_*$ is a morphism of $\ZM[\G]$-algebras.}
\endequat
If $h \in \HC(W,S,\ph)$, then 
\equat\label{rho bar}
\overline{\r_*(h)}=\r_*(\overline{h}).
\endequat
Moreover, if $\s : \G' \to \G''$ is another morphism of abelian groups, then 
\equat\label{covariance rho}
(\s \circ \r)_* = \s_* \circ \r_*.
\endequat
The next lemma is an obvious consequence:

\bigskip

\begin{lem}\label{rho bij}
The morphism of groups $\r$ is injective (respectively surjective, 
respectively bijective) if and only if the morphism of algebras $\r_*$ is.
\end{lem}

\bigskip

\subsection{Generic Hecke algebra} 
Let $R$ be the group algebra $\ZM[\ZM[\Sba]]$. 
Let $i : \Sba \to \ZM[\Sba]$ be the canonical map. 
The Hecke algebra $\HC(W,S,i)$ will then be denoted by 
$\HCB(W,S)$: it is called the {\it generic Hecke algebra}. 

It is universal in the following sense: if we identify the map 
$\ph : \Sba \to \G$ with the morphism of groups $\ph : \ZM[\Sba] \to \G$, 
then the Hecke algebra $\HC(W,S,\ph)$ is now equipped with a canonical 
morphism of $R$-algebras $\ph_* : \HCB(W,S) \to \HC(W,S,\ph)$, so that  $\HC(W,S,\ph)$ is a specialization of $\HCB(W,S)$, i.e.
$$\HC(W,S,\ph)=\ZM[\G] \otimes_R \HC(W,S).$$

\bigskip

\section{Kazhdan-Lusztig cells}

\medskip

\subsection{Kazhdan-Lusztig basis} 
To define the Kazhdan-Lusztig basis, we shall need 
to work all along this paper 
under the following hypothesis:

\medskip

\begin{quotation}
\noindent{\bf Hypothesis and notation.} {\it 
Until the end of this paper, we assume that $\G$ is endowed with a total 
order $\le$ compatible with the group law. We denote respectively by 
$\G_{>0}$, $\G_{\geqslant 0}$, $\G_{< 0}$ and $\G_{\leqslant 0}$ 
the set of positive, non-negative, negative, non-positive elements 
of $\G$.} 
\end{quotation}

\medskip

If $E$ is any subset of $\G$, we set 
$\ZM[E]=\DS{\mathop{\oplus}_{\g \in E}} \ZM e^\g$. For instance, 
$\ZM[\G_{\leqslant 0}]$ is a subring of $\ZM[\G]$ and 
$\ZM[\G_{<0}]$ is an ideal of $\ZM[\G_{\leqslant 0}]$. 
Let  
$$\HC_{<0}(W,S,\ph) = \mathop{\oplus}_{w \in W} \ZM[\G_{< 0}] ~T_w.$$
Then, if $w \in W$, there exists \cite[Theorem 5.2]{lusztig} a unique element 
$C_w \in \HC(W,S,\ph)$ such that 
$$\begin{cases}
\overline{C}_w = C_w,& \\
C_w \equiv T_w \mod \HC_{<0}(W,S,\ph).&
\end{cases}$$
Again, if necessary, the element $C_w$ will be denoted by $C_w^\ph$. 

\def\pre#1#2{\leqslant_{#1}^{#2}}

The family $(C_w)_{w \in W}$ is a $\ZM[\G]$-basis of 
$\HC(W,S,\ph)$ (called the {\it Kazhdan-Lusztig basis} 
\cite[Theorem 5.2]{lusztig}). If $x$, $y \in W_n$, then we shall write 
$x \gauche y$ (respectively $x \droite y$, respectively $x \bilatere y$) 
if there exists $h \in \HC_n$ such that the coefficient of $C_x$ 
in the decomposition of $hC_y$ (respectively $C_y h$, respectively 
$hC_y$ or $C_y h$) is non-zero. We denote by $\pre{L}{\ph}$ 
(respectively $\pre{R}{\ph}$, respectively $\pre{LR}{\ph}$) the transitive closure 
of $\gauche$ (respectively $\droite$, respectively $\bilatere$). 
Then $\pre{L}{\ph}$, $\pre{R}{\ph}$ and $\pre{LR}{\ph}$ are preorders on $W$ 
and we denote respectively by $\sim_L^\ph$, $\sim_R^\ph$ and $\sim_{LR}^\ph$ 
the associated equivalence relations \cite[Chapter 8]{lusztig}. An 
equivalence class for $\sim_L^\ph$ (respectively $\sim_R^\ph$, respectively $\sim_{LR}^\ph$) 
is called a {\it left} (respectively {\it right}, respectively {\it two-sided}) 
{\it cell} (for $(W,S,\ph)$). We recall the following result 
\cite[\S 8.1]{lusztig}: if $x$, $y \in W_n$, then
\equat\label{siml simr}
x \sim_L^\ph y \Longleftrightarrow x^{-1} \sim_R^\ph y^{-1}.
\endequat

If $w \in W$, and if $? \in \{L, R, LR\}$, 
we set
$$\Cell_?^\ph(w)=\{x \in W~|~x \sim_?^\ph w\}.$$

\bigskip

\subsection{Cell representations} 
Let $? \in \{L,R,LR\}$. For simplification, we define a {\it $?$-ideal} 
to be a left ideal if $?=L$, a right ideal if $? =R$ and a two-sided 
ideal if $?=LR$. Similarly, if $A$ is a ring, a {\it $?$-module} 
(or an $A$-$?$-module) is a left $A$-module if $?=L$, 
a right $A$-module if $?=R$ and an $(A,A)$-bimodule if $?=LR$. 

If $C$ is a cell in $W$ for $\sim_?^\ph$, we set, 
following \cite[\S 8.3]{lusztig}, 
$$\HC(W,S,\ph)_{\pre{?}{\ph} C}=\mathop{\oplus}_{w \pre{?}{\ph} C} \ZM[\G]~C_w,
\quad
\HC(W,S,\ph)_{<_?^\ph C}=\mathop{\oplus}_{w <_?^\ph C} \ZM[\G]~C_w$$
$$M_C^{?,\ph}=\HC(W,S,\ph)_{\pre{?}{\ph} C}/\HC(W,S,\ph)_{<_?^\ph C}.
\leqno{\text{and}}$$
Then, by definition, $\HC(W,S,\ph)_{\pre{?}{\ph} C}$ and 
$\HC(W,S,\ph)_{<_?^\ph C}$ are $?$-ideals and $M_C^{?,\ph}$ 
is an $\HC(W,S,\ph)$-$?$-module. Note that $M_C^{?,\ph}$ is a 
free $\ZM[\G]$-module with basis the image of $(C_w)_{w \in C}$.

Let $\aug : \ZM[\G] \to \ZM$ be the augmentation morphism and view $\ZM$ 
(or $\QM$) as a $\ZM[\G]$-algebra through $\aug$. We have an isomorphism 
of rings $\ZM \otimes_{\ZM[\G]} \HC(W,S,\ph) \simeq \ZM W$. Let 
$\ZM M_C^{?,\ph}$ (respectively $\QM M_C^{?,\ph}$) be the $\ZM W$-$?$-module 
(respectively $\QM W$-$?$-module) $\ZM \otimes_{\ZM[\G]} M_C^{?,\ph}$ 
(respectively $\QM \otimes_{\ZM[\G]} M_C^{?,\ph}$). 

If $W$ is {\it finite}, we denote by 
$\chi_C^{?,\ph}$ the character of $\QM M_C^{?,\ph}$ (if $?=LR$, 
it is a character of $W \times W$). 

\bigskip

\subsection{Strictly increasing morphisms} 
The main goal of this paper is to study how the partition into 
cells behaves whenever the datum $(\G,\le,\ph)$ varies. The first 
easy remark is that this partition does not change by composition 
with a strictly increasing morphism of groups:

\bigskip

\begin{prop}\label{strictement croissant}
Let $\G'$ be a totally ordered abelian group, let 
$\r : \G \to \G'$ be a strictly increasing morphism of groups and let 
$? \in \{L,R,LR\}$. Then:
\begin{itemize}
 \itemth{a} $\r_*$ is injective.

\itemth{b} If $w \in W$, then $\r_*(C_w^\ph)=C_w^{\r \circ \ph}$. 

\itemth{c} The relations $\le_?^\ph$ et $\le_?^{\r \circ \ph}$ coincide 
(as well as the relations $\sim_?^\ph$ et $\sim_?^{\r \circ \ph}$). 

\itemth{d} If $w \in W$, then  $\Cell_?^\ph(w)=\Cell_?^{\r \circ \ph}(w)$. 
\end{itemize}
\end{prop}

\begin{proof}
The injectivity of $\r_*$ follows from the fact that $\r$ 
itself is injective (see Lemma \ref{rho bij}). 
Hence (a) holds. Let us show (b). Let $C=\r_*(C_w^\ph)$. 
Then, by \ref{rho bar}, we get $\Cov=C$. The fact that $\r$ is 
strictly increasing implies that $\r_*(\ZM[\G_{<0}]) 
\subseteq \ZM[\G_{<0}^\prime]$, so that 
\equat\label{rho positif}
\r_*(\HC_{<0}(W,S,\ph)) \subseteq \HC_{< 0}(W,S,\r \circ \ph).
\endequat
Therefore, $C - T_w^{\r \circ \ph} \in \HC_{<0}(W,S,\r \circ \ph)$. 
So $C=C_w^{\r \circ \ph}$ by the characterization of the elements 
of the Kazhdan-Lusztig basis. 

Now (c) follows immediately from (a) and (b), while (d) 
is an immediate consequence of (c).
\end{proof}

\bigskip

\subsection{Changing signs}
We shall show in this subsection that changing the sign of some 
values of the function $\ph$ has no effect on the partition into 
cells (and that it has little effect on the cell representations). 
For this, we shall work under the following hypothesis:

\medskip

\begin{quotation}
\noindent{\bf Hypothesis and notation.} 
{\it In this subsection, and only in this subsection, 
we fix a partition $S=S_+ \dotcup S_-$ of $S$ such that no element of 
$S_+$ is conjugate to an element of $S_-$. Let $\ph' : \Sba \to \G$ 
be the (well-defined) map such that
$$\ph'(\sba)=\begin{cases}
\ph(\sba) & \text{if $s \in S_+$,}\\
-\ph(\sba) & \text{if $s \in S_-$.}
\end{cases}$$}
\end{quotation}

\medskip

If $w \in W$, we set  
$$\ell_\pm(w)=\sum_{\o \in \Sba_\pm} \ell_\o(w),$$
so that $\ell(w)=\ell_+(w) + \ell_-(w)$. Let 
$\th : \HC(W,S,\ph) \longto \HC(W,S,\ph')$ be the unique 
$\ZM[\G]$-linear map such that 
$$\th(T_w^\ph)=(-1)^{\ell_-(w)} T_w^{\ph'}.$$
An elementary computation shows that:

\bigskip

\begin{prop}\label{signe}
The map $\th : \HC(W,S,\ph) \longto \HC(W,S,\ph')$ 
is an isomorphism of $\ZM[\G]$-algebras. It sends $\HC_{<0}(W,S,\ph)$ 
onto $\HC_{<0}(W,S,\ph')$. If $h \in \HC(W,S,\ph)$, then $\th(\overline{h})=\overline{\th(h)}$.  
Consequently, if $w \in W$, then 
$$\th(C_w^\ph) = (-1)^{\ell_-(w)} C_w^{\ph'}.$$
\end{prop}

\bigskip

\begin{coro}\label{signe cellules}
If $? \in \{L,R,LR\}$, then the relations $\pre{?}{\ph}$ and 
$\pre{?}{\ph'}$ coincide. Similarly, the relations 
$\sim_?^\ph$ and $\sim_?^{\ph'}$ coincide. 
\end{coro}

\bigskip

Let $\g : W \to \{1,-1\}$ denote the unique linear character of $W$ such that 
$$\g(s)=
\begin{cases}
1 & \text{if $s \in S_+$},\\
-1 & \text{if $s \in S_-$}.
\end{cases}$$
In other words, $\g(w)=(-1)^{\ell_-(w)}$. 
Let $\ZM_\g$ (respectively $\ZM_{\g \boxtimes \g}$) 
denote the (left or right) $\ZM W$-module (respectively the 
$(\ZM W,\ZM W)$-bimodule) of $\ZM$-rank $1$ on which 
$W$ (respectively $W \times W$) acts through the linear character 
$\g$ (respectively the linear character $\g \boxtimes \g : W \times W \to \{1,-1\}$, 
$(w,w') \mapsto \g(w)\g(w')$). If $C$ is a cell for $\sim_?^\ph$, 
then an easy computation using the isomorphism $\th$ shows that 
\equat\label{module signe}
\ZM M_C^{?,\ph'} \simeq 
\begin{cases}
\ZM_\g \otimes_\ZM \ZM M_C^{?,\ph} & \text{if $?\in \{L,R\}$,}\\
\ZM_{\g \boxtimes \g} \otimes_\ZM \ZM M_C^{?,\ph} & \text{if $?=LR$.}
\end{cases}
\endequat

\bigskip

\subsection{Zero values of ${\boldsymbol{\ph}}$} 
It follows from Corollary \ref{signe cellules} that the computation 
of Kazhdan-Lusztig cells can be reduced to the case where $\ph$ 
has non-negative values. We shall study here what happens when 
$\ph$ vanishes on some elements of $S$. We shall work under 
the following hypothesis:

\medskip

\begin{quotation}
\noindent{\bf Hypothesis and notation.} 
{\it In this subsection, and only in this subsection, we fix a partition 
$S=I \dotcup J$ of $S$ such that no element of $I$ is conjugate to an 
element of $J$. Let $W_I$ denote the subgroup of $W$ generated by $I$ and 
we set 
$$\Jti=\{wtw^{-1}~|~w \in W_I\text{ and }t \in J\}.$$
Let $\Wti$ be the subgroup of $W$ generated by $\Jti$. 
We shall also assume that
$$\text{\it if $s \in I$, then $\ph(s)=0$.}\leqno{\hskip1.3cm(*)}$$
If $w \in W$, let $\ell_I(w)$ and $\ell_J(w)$ denote respectively 
the number of occurrences of elements of $I$ and $J$ in 
a reduced expression of $w$.}
\end{quotation}

\medskip

By a result of Gal \cite[Proposition 2.1]{gal} (see also 
\cite[Theorem 1.1]{bonnafe dyer}), $(\Wti,\Jti)$ is a Coxeter system and 
\equat\label{dyer bis}
W=W_I \ltimes \Wti
\endequat
We shall denote by $\ellt : \Wti \to \NM$ the length function on 
$\Wti$ with respect to $\Jti$. In fact \cite[Corollary 1.3]{bonnafe dyer},
if $w \in \Wti$, then
\equat\label{ltilde}
\ellt(w)=\ell_J(w).
\endequat
If $\tti \in \Jti$, we denote by $\nu(\tti)$ the unique element of 
$J$ such that $\tti$ is $W_I$-conjugate to $\nu(\tti)$ 
(see \cite[(1.3)]{bonnafe dyer}). We set 
$$\pht(\tti)=\ph(\nu(\tti)).$$
By \cite[Corollary 1.4]{bonnafe dyer} that, if 
$\tti$ and $\tti'$ are two elements of $\Jti$ which are conjugate 
in $\Wti$, then 
\equat\label{pht}
\pht(\tti)=\pht(\tti').
\endequat
This shows that we can define a Hecke $\ZM[\G]$-algebra 
$\HC(\Wti,\Jti,\pht)$. The group $W_I$ acts on $\Wti$ and stabilizes 
$\Jti$ and the map $\pht$, so it acts naturally on the Hecke algebra 
$\HC(\Wti,\Jti,\pht)$. We can then define the semidirect product of algebras 
$$W_I \ltimes \HC(\Wti,\Jti,\pht).$$
It is a $\ZM[\G]$-algebra with $\ZM[\G]$-basis 
$(x \cdot T_w^\pht)_{x \in W_I,w \in \Wti}$. 
Finally, let
$$\HCt=\mathop{\oplus}_{w \in \Wti} \ZM[\G]~T_w^\ph \quad \subseteq \HC(W,S,\ph).$$

\begin{prop}\label{semi-direct}
Recall that $\ph(s)=0$ if $s \in I$. Then:
\begin{itemize}
\itemth{a} $\HCt$ is a sub-algebra of $\HC(W,S,\ph)$.

\itemth{b} The unique $\ZM[\G]$-linear map 
$\th : W_I \ltimes \HC(\Wti,\Jti,\pht) \longto \HC(W,S,\ph)$ that sends 
$x \cdot T_w^\pht$ on $T_{xw}^\ph$ ($x \in W_I$, $w \in \Wti$) 
is an isomorphism of $\ZM[\G]$-algebras. It sends $\HC(\Wti,\Jti,\pht)$ 
isomorphically onto $\HCt$.

\itemth{c} If $x \in W_I$ and $w \in W$, then 
$C_x^\ph=T_x^\ph$, $C_{xw}^\ph=T_x^\ph C_w^\ph$, 
$C_{wx}^\ph=C_w^\ph T_x^\ph$. If moreover 
$w \in \Wti$, then $\th(C_w^\pht)=C_w^\ph$.
\end{itemize}
\end{prop}

\begin{proof}
Since $\ph(I)=\{0\}$, we have, for all $s \in I$, 
$$(T_s^\ph)^2 = 1\quad\text{and}\quad \overline{T_s^\ph} = T_s^\ph.$$
So it follows by an easy induction argument that:
\equat\label{lj}
\text{\it If $\ell_J(xy)=\ell_J(x)+\ell_J(y)$, then 
$T_x^\ph T_y^\ph = T_{xy}^\ph$.}
\endequat
In particular, if $x \in W_I$ and $w \in W$, then 
$T_x^\ph T_w^\ph=T_{xw}^\ph$ et $T_w^\ph T_x^\ph=T_{wx}^\ph$. 
We deduce that $C_x^\ph=T_x^\ph$, $C_{xw}^\ph=T_x^\ph C_w^\ph$, 
$C_{wx}^\ph=C_w^\ph T_x^\ph$. This proves the first assertion of (c). 

\medskip

Consequently, if $x \in W_I$ and $t \in J$, then 
$$T_{xtx^{-1}}^\ph=T_x^\ph T_t^\ph (T_x^\ph)^{-1}$$
and so, if we set $\tti=xtx^{-1}$, we have 
$$(T_\tti^\ph - e^{\pht(\tti)})(T_{\tti}^\ph+e^{-\pht(\tti)})=0.$$
To show (a) and (b), it only remains to show the following: 
if $w$ and $w'$ are two elements of $\Wti$ such that 
$\ellt(ww')=\ellt(w)+\ellt(w')$ (here, $\ellt$ denotes the length 
function on $\Wti$ associated to $\Jti$), then 
$$T_w^\ph T_{w'}^\ph =T_{ww'}^\ph.$$
But this follows from \ref{ltilde} and \ref{lj}.

\medskip

In order to show the last assertion of (c), it is sufficient to notice that, 
if $h \in \HC(\Wti,\Jti,\pht)$, then $\th(\overline{h})=\overline{\th(h)}$, 
and to use the characterization of the elements of the Kazhdan-Lusztig basis.
\end{proof}

\begin{coro}\label{ordre nul}
Let $a$, $b \in W_I$, $x$, $y \in \Wti$. 
Then:
\begin{itemize}
\itemth{a} $ax \pre{L}{\ph} by$ if and only if $x \pre{L}{\pht} y$. 

\itemth{b} $xa \pre{R}{\ph} yb$ if and only if $x \pre{LR}{\pht} y$.

\itemth{c} $ax \pre{LR}{\ph} yb$ if and only if there exists $c \in W_I$ such 
that $x \pre{LR}{\pht} cyc^{-1}$.
\end{itemize}
\end{coro}

\begin{proof}
(b) and (c) follow easily from (a). So let us prove (a). 
It is sufficient to show that $ax \stackrel{L,\ph}{\longleftarrow} by$ 
if and only if $x \stackrel{L,\pht}{\longleftarrow} y$. 

First assume that $ax \stackrel{L,\ph}{\longleftarrow} by$. 
Then there exists $w \in W$ such that $C_{ax}^\ph$ occurs 
in the decomposition of $T_w^\ph C_{by}^\ph$. Let us write 
$w=cz$, with $c \in W_I$ and $z \in \Wti$. Then 
$$T_w^\ph C_{by}^\ph = T_{cb}^\ph T_{b^{-1}zb}^\ph C_y^\ph.$$
This shows that $cb=a$ and that $C_x^\ph$ occurs in 
the decomposition of $T_{b^{-1}zb}^\ph C_y^\ph$. 
By applying the Proposition \ref{semi-direct} (and the isomorphism $\th$), 
we get that $x \stackrel{L,\pht}{\longleftarrow} y$. 

Conversely, assume that $x \stackrel{L,\pht}{\longleftarrow} y$. 
Then there exists $w \in \Wti$ such that $C_x^\pht$ occurs 
in the decomposition of $T_w^\pht C_y^\pht$. 
By applying the Proposition \ref{semi-direct} (and the isomorphism $\th$), 
we get that $C_{ax}^\ph$ occurs in the decomposition of 
$T_{awb^{-1}}^\ph C_{by}^\ph$. So $ax \stackrel{L,\ph}{\longleftarrow} by$. 
\end{proof}

We can immediately deduce the following corollary:

\begin{coro}\label{KL semi-direct}
Assume that $\ph(I) =\{0\}$. 
Then the left (respectively right, respectively two-sided) cells 
for $(W,S,\ph)$ are of the form $W_I \cdot C$ (respectively $C \cdot W_I$, 
respectively $W_I \cdot C \cdot W_I$), 
where $C$ is a left (respectively right, respectively two-sided) cell for  $(\Wti,\Jti,\pht)$.
\end{coro}

\bigskip

At the level of cell representations, we get:

\begin{coro}\label{induit zero}
If $C$ if a left cell for $(\Wti,\Jti,\pht)$, then 
$$\ZM M_{W_I\cdot C}^{L,\ph} \simeq \Ind_{\Wti}^W \ZM M_C^{L,\pht}.$$
\end{coro}

\begin{proof}
This follows easily from the Proposition \ref{semi-direct} (and its proof) 
and the Corollaries \ref{ordre nul} and  \ref{KL semi-direct}.
\end{proof}

\bigskip

\section{Conjectures about cells}

\bigskip

\begin{quotation}
{\it From now on, and until the end of this paper, we assume that 
$S$ is finite. We shall use the notion, notation and results of the Appendix: 
positive subsets of $\ZM[\Sba]$, topology on $\positif(\ZM[\Sba])$, 
hyperplane arrangements, $\Pos(\ph)$, $\UC(\l)$, $\HC_\l$...}
\end{quotation}

\bigskip

As explained before, the main aim of this paper is to study 
the behaviour of the relations $\sim_?^\ph$ as the datum 
$(\G,\le,\ph)$ varies. The first step is to show that 
the relations $\pre{?}{\ph}$ (and $\sim_?^\ph$) depend only on the positive 
subset $\Pos(\ph) \in \positif(\ZM[\Sba])$ (see Corollary \ref{phi X}). 
This allows us to define relations $\pre{?}{X}$ and 
$\sim_?^X$ for $X \in \positif(\ZM[\Sba])$. 
Our conjectures are then about the behaviour of the relations 
$\sim_?^X$ whenever $X$ runs over $\positif(\ZM[\Sba])$: 
this involves the topology of this set together with the 
notion of hyperplane arrangements (see the Appendix). 

\bigskip

\subsection{Positive subsets of ${\boldsymbol{\ZM[\Sba]}}$} 
Let $X$ be a positive subset of $\ZM[\Sba]$. 
Let $\G_X = \ZM[\Sba]/(X \cap (-X))$ and let $\leqslant_X$ be the 
total order on $\G_X$ defined in \ref{ordre X}. Let $\ph_X : \Sba \to \G_X$ 
be the canonical map (it is the composition of the natural map 
$\Sba \to \ZM[\Sba]$ with the canonical map $\can_X : \ZM[\Sba] \to \G_X$
of the Appendix). 
For simplification, the relations $\le_?^{\ph_X}$ et $\sim_?^{\ph_X}$ will 
be denoted by $\le_?^X$ and $\sim_?^X$. Similarly, if $w \in W$, 
we will denote by $\Cell_?^X(w)$ the subset $\Cell_?^{\ph_X}(w)$ and, 
if $C$ is a $?$-cell for $(W,S,\ph_X)$, we denote by 
$\ZM M_C^{?,X}$ the $\ZM W$-module $\ZM M_C^{?,\ph_X}$ 
(and, if $W$ is {\it finite}, we denote by $\chi_C^{?,X}$ the character 
$\chi_C^{?,\ph_X}$).
The next proposition and its corollary show that the family 
$\bigl((\G_X,\le_X,\ph_X)\bigr)_{X \in \positif(\ZM[\Sba])}$ 
is essentially exhaustive.

\bigskip

\begin{prop}
Let $X=\Pos(\ph)$. Then there exists a unique morphism of groups 
$\phba : \G_X \to \G$ such that $\ph = \phba \circ \ph_X$. 
This morphism $\phba$ is strictly increasing. 
\end{prop}

\bigskip

\begin{proof} 
Indeed, $\Ker \ph = X \cap (-X)$ so the existence, the unicity and the injectivity of $\phba$ are clear. On the 
other hand, if $\g$, $\g' \in \G_X$ are such that $\g \le \g'$ and if 
$\l \in \ZM[\Sba]$ is a representative of $\g'-\g$, then $\l \in X$ 
by \ref{relation independante}. 
So $\ph(\l) \ge 0$. In other words $\phba(\g'-\g) \ge 0$  
that is, $\phba(\g) \le \phba(\g')$. So $\phba$ is increasing: 
the fact that it is strictly increasing follows from the injectivity. 
\end{proof}

\bigskip

\begin{coro}\label{phi X}
Let $X=\Pos(\ph)$ and let $? \in \{L,R,LR\}$. Then:
\begin{itemize}
\itemth{a} The relations 
$\le_?^\ph$ and $\le_?^X$ are equal (as well as the relations 
$\sim_?^\ph$ and $\sim_?^X$). 

\itemth{b} If $w \in W$, then 
$\Cell_?^\ph(w)=\Cell_?^X(w)$.

\itemth{c} If $C=\Cell_L^\ph(w)$, then $\ZM M_C^\ph \simeq \ZM M_C^X$. 
\end{itemize}
\end{coro}

\bigskip

Thanks to Corollary \ref{phi X}, studying the behaviour of 
the relations $\sim_?^\ph$ as the weight function $\ph$ varies 
is equivalent to studying the behaviour of relations $\sim_?^X$ 
as $X$ runs over $\positif(\ZM[\Sba])$. 
The following remark can be useful for switching from 
one point of view to the other:

\bigskip

\remark{phi X traduction}
Let $\l \in \ZM[\Sba]$ and $X \in \positif(\ZM[\Sba])$. Then the 
following hold:
\begin{itemize}
\itemth{a} $X \in \UC(\l)$ if and only if $\ph_X(\l) <_X 0$.
\itemth{b} $X \in \HC_\l$ if and only if $\ph_X(\l) = 0$.
\itemth{c} $X \in \overline{\UC(\l)}$ if and only if 
$\ph_X(\l) \pre{X}{} 0$.\finl
\end{itemize}

\bigskip

In the same spirit, the Corollary \ref{signe cellules} can be 
translated into the Proposition \ref{symetrie} below. We first 
need some notation. 
If $\o \in \Sba$, we denote by $\t_\o$ the $\ZM$-linear symmetry 
on $\ZM[\Sba]$ such that $\t_\o(\o)=-\o$ and $\t_\o(\o')=\o'$ if 
$\o' \neq \o$. It is an automorphism of $\ZM[\Sba]$: it induces 
an homeomorphism $\t_\o^*$ of $\positif(\ZM[\Sba])$.


\bigskip

\begin{prop}\label{symetrie}
If $\o \in \Sba$, if $X \in \positif(\ZM[\Sba])$ and if 
$? \in \{L,R,LR\}$, then the relations $\sim_?^X$ and 
$\sim_?^{\t_\o^*(X)}$ coincide.
\end{prop}

\bigskip

\begin{proof}
The map $\t_\o$ induces a strictly increasing isomorphism 
$$\t_\o : \G_{\t_\o^*(X)} \stackrel{\sim}{\longrightarrow} \G_X.$$
Therefore, the relations $\sim_?^{\ph_{\t_\o^*(X)}}$ 
and $\sim_?^{\t_\o \circ \ph_{\t_\o^*(X)}}$ coincide 
(see the Proposition \ref{strictement croissant}). 
Now, let $\ph_X'=\t_\o \circ \ph_{\t_\o^*(X)} : \Sba \to \G_X$. 
But
$$\ph_X'(\o')=\begin{cases}
\ph_X(\o') & \text{if $\o \neq \o'$,}\\
-\ph_X(\o) & \text{if $\o=\o'$.}
\end{cases}$$
So the result follows from Corollary \ref{signe cellules}.
\end{proof}

\bigskip

\subsection{Conjectures} 
A rational hyperplane arrangement $\AG$ in $\positif(\ZM[\Sba])$ is called 
{\it complete} if $\HC_\o \in \AG$ for all $\o \in \Sba$. Recall that
\eqna
\HC_\o&=&\{X \in \positif(\ZM[\Sba])~|~\o \in X \cap -X\}\\
&=& \{X \in \positif(\ZM[\Sba])~|~\ph_X(\o)=0\}.
\endeqna
If $\AG$ is a complete arrangement and if $\FC$ is an $\AG$-facet, 
we denote by $W_\FC$ the standard parabolic subgroup generated 
by $\bigcup_{\o \in \Sba, \FC \subseteq \HC_\o} \o$. As in the 
Appendix, we denote by $\facets(\AG)$ and $\chambres(\AG)$ the set 
of $\AG$-facets and $\AG$-chambers respectively. 

We denote by $\relations(W)$ the set of equivalence relations on $W$. 
If $\RC$ and $\RC'$ are two equivalence relations on $W$, we write 
$\RC \infspe \RC'$ if $\RC$ is finer than $\RC'$ (i.e. if $x \RC y$ 
implies $x \RC' y$). The poset $(\relations(W),\infspe)$ is 
a lattice (i.e. every pair of two elements has a supremum and an 
infimum): $\sup(\RC,\RC')$ is the equivalence relation generated 
by $\RC$ and $\RC'$ while $\inf(\RC,\RC')$ is the equivalence relation 
$\RC''$ defined by $x \RC'' y$ if and only if $x \RC y$ and 
$x \RC' y$ (for $x$, $y \in W$). 
Finally, if $? \in \{L,R,LR\}$ and if $H$ is a subgroup of $W$, 
we denote by $\trans_?^H$ the equivalence relation on $W$ defined by the 
action of $H$ (or $H \times H$) by translation on $W$:
$$x ~\trans_?^H~ y \Leftrightarrow
\begin{cases}
\exists~w \in H,~y=wx,&\text{if $? = L$,}\\
\exists~w \in H,~y=xw,&\text{if $? = R$,}\\
\exists~w, w' \in H,~y=wxw',&\text{if $? = LR$.}
\end{cases}
$$

\bigskip

\noindent{\bf Conjecture A.} 
{\it Assume that $S$ is finite. 
Then there exists a finite complete rational hyperplane arrangement 
$\AG$ in $\positif(\ZM[\Sba])$ satisfying the following properties 
(for all $? \in \{L,R,LR\}$):
\begin{itemize}
\itemth{a} If $X$ and $Y$ are two positive subsets of $\ZM[\Sba]$ 
belonging to the same $\AG$-facet $\FC$, then the relations 
$\sim_?^X$ and $\sim_?^Y$ coincide (we will denote it by $\sim_?^\FC$).

\itemth{b} Let $\FC$ be an $\AG$-facet. 
Then the cells for the relation 
$\sim_?^\FC$ are the minimal subsets $C$ of $W$ satisfying the following 
conditions:
\begin{itemize}
\itemth{b1} For every chamber $\CC$ such that $\FC \subseteq \overline{\CC}$, 
$C$ is a union of cells for $\sim_?^\CC$;

\itemth{b2} $C$ is stable by translation by $W_\FC$ 
(on the left if  $?=L$, on the right if $?=R$, on the left and on the right 
if $?=LR$).
\end{itemize}
\end{itemize}}

\bigskip

\remark{supremum} 
With the notation of the above Conjecture, the statement (b) is equivalent 
to the following one:
\begin{itemize}
\itemth{b'} {\it Let $\FC$ be an $\AG$-facet. Then 
$$\sim_?^\FC = \sup\Bigl(\trans_?^{W_\FC},
\sup_{\substack{\CC \in \chambres(\AG) \\ \FC \subseteq \overline{\CC}}} 
\sim_?^\CC\Bigr).$$}
\end{itemize}
Here, the suprema are computed in the set $\relations(W)$.\finl

\medskip

\remark{tr} 
The statement (b2) in the Conjecture A comes from 
Corollary \ref{KL semi-direct}. It is necessary, as it can 
already be seen in the case where $|S|=1$.\finl

\bigskip

The previous remark (together with the Corollary \ref{upper} of the
Appendix) implies immediately the following proposition (which 
``justifies'' the title of the paper):

\begin{prop}\label{upper rel}
If the Conjecture A holds for $(W,S)$, then the map 
$$\fonctio{\positif(\ZM[\Sba])}{\relations(W)}{X}{\sim_?^X}$$
is upper semicontinuous (for all $? \in \{L,R,LR\}$). 
\end{prop}

\begin{coro}\label{upper cell}
Let $? \in \{L,R,LR\}$, let $w \in W$ and let $\PC(W)$ denote the 
power set of $W$. 
If the Conjecture A holds for $(W,S)$, then the map 
$$\fonctio{\positif(\ZM[\Sba])}{\PC(W)}{X}{\Cell_?^X(w)}$$
is upper semicontinuous. 
\end{coro}

\bigskip

\subsection{Another form of Conjecture A} 
Let us now give a translation of this conjecture in terms of 
maps $\ph : \Sba \to \G$. For this, let $\SGN =\{+,-,0\}$ and 
let $\EC$ be a finite set of elements in $\ZM[\Sba]\setminus \{0\}$. Let 
$\SGN^\EC$ denote the set of maps $\EC \to \SGN$. If 
$X \in \positif(\ZM[\Sba])$, we set:
$$\fonction{\sgn_\EC(X)}{\EC}{\SGN}{\l}{\begin{cases}
+ & \text{if $\l \not\in -X$,}\\
0 & \text{if $\l \in X \cap (-X)$,}\\
- & \text{if $\l \not\in X$.}
\end{cases}}$$
This defines a map 
$$\sgn_\EC : \positif(\ZM[\Sba]) \longto \SGN.$$
Similarly, we set 
$$\fonction{\sgn_\EC(\ph)}{\EC}{\SGN}{\l}{\begin{cases}
+ & \text{if $\ph(\l) > 0$,}\\
0 & \text{if $\ph(\l) = 0$,}\\
- & \text{if $\ph(\l) < 0$.}
\end{cases}}$$
In other words,
\equat\label{sgn}
\sgn_\EC(X)=\sgn_\EC(\ph_X)\quad\text{and}
\quad \sgn_\EC(\ph)=\sgn_\EC(\Pos(\ph)).
\endequat
Let $\AG_\EC$ be the finite rational hyperplane arrangement 
$\{\HC_\l~|~\l \in \EC\}$. By the very definition of facets, 
the $\AG_\EC$-facets are exactly the non-empty fibers of the map
$\sgn_\EC : \positif(\ZM[\Sba]) \longto \SGN^\EC$. 
If $\FC$ is an $\AG_\EC$-facet, then we denote by $\sgn_\EC(\FC)$ 
the element $\sgn_\EC(X) \in \SGN^\EC$, where $X$ is some (or any) element 
of $\FC$. 

We now endow $\SGN$ with the unique partial order $\infspe$ such that 
$0 \infspe +$, $0 \infspe -$ and $+$ and $-$ are not comparable. 
This defines a partial order on $\SGN^\EC$ (componentwise) which 
is still denoted by $\infspe$. Then, by Proposition \ref{facettes}, 
we have, for all facets $\FC$ and $\FC'$, 
\equat\label{ordre sgn}
\text{\it $\FC \subseteq \overline{\FC}'$ if and only 
if $\sgn_\EC(\FC) \infspe \sgn_\EC(\FC')$.}
\endequat
(Note that this implies, by Corollary \ref{upper}, that 
the map $\sgn_\EC$ is lower semicontinuous.) 
Finally, we say that the map $\ph$ is {\it $\EC$-open} if $\Pos(\ph)$ 
belongs to an $\AG_\EC$-chamber: equivalently, $\ph$ is $\EC$-open 
if $\sgn_\EC(\ph)(\l) \neq 0$ for all $\l \in \EC \setminus\{0\}$.

By the above discussion, we get that 
the next Conjecture is clearly equivalent to Conjecture A 
(here, we denote by $W_\ph$ the parabolic subgroup generated by 
$\{s \in S~|~\ph(s)=0\}$).

\bigskip

\noindent{\bf Conjecture A'.} 
{\it Assume that $S$ is finite. Then there exists a finite set 
$\EC$ in $\ZM[\Sba]\setminus\{0\}$ containing $\Sba$ and 
such that, for all $? \in \{L,R,LR\}$, we have:
\begin{itemize}
\itemth{a} If $\G$ and $\G'$ are two abelian ordered groups and if 
$\ph : \Sba \to \G$ and $\ph' : \Sba \to \G'$ are two maps such 
that $\sgn_\EC(\ph)=\sgn_\EC(\ph')$, then the relations $\sim_?^\ph$ 
and $\sim_?^{\ph'}$ coincide.

\itemth{b} If $\G$ is a totally ordered abelian group and if 
$\ph : \Sba \to \G$ is a map, then the cells for $\sim_?^\ph$ are the minimal 
subsets $C$ of $W$ such that:
\begin{itemize}
\itemth{b1} For all totally ordered abelian group $\G'$ and for all 
$\EC$-open maps $\ph' : \Sba \to \G'$ such that 
$\sgn_\EC(\ph) \infspe \sgn_\EC(\ph')$, 
$C$ is a union of cells for $\sim_?^{\ph'}$.

\itemth{b2} $C$ is stable by translation by $W_\ph$ (on the left if 
$?=L$, on the right if $?=R$, on the left and on the right if $?=LR$).
\end{itemize}
\end{itemize}}

\bigskip

An element $\l \in \ZM[\Sba]$ is called {\it reduced} 
if $\l \neq 0$ and $\ZM[\Sba]/\ZM \l$ is torsion-free. If $\HC$ is a rational 
hyperplane, then there exist only two reduced elements $\l \in \ZM[\Sba]$ 
such that $\HC=\HC_\l$ (one is the opposite of the other). 
A subset $\EC$ of $\ZM[\Sba]$ is called {\it reduced} if all 
its elements are reduced. It is called {\it complete} if 
$\Sba \subseteq \EC$. It is called {\it symmetric} if $\EC=-\EC$. 
If $\AG$ is a rational hyperplane arrangement, then there exists 
a unique reduced symmetric subset $\EC$ of $\ZM[\Sba]$ such that 
$\AG=\AG_\EC$. In this case, $\AG$ is complete if and only if $\EC$ is. 

\bigskip

\subsection{Essential hyperplanes} 
The Remark \ref{supremum} implies easily the following

\begin{prop}\label{essentiellement}
If the Conjecture A holds for $(W,S)$ and two finite complete 
rational hyperplane arrangements $\AG$ and $\AG'$ of 
$\positif(\ZM[\Sba])$, then it holds for $(W,S)$ 
and for the finite complete arrangement $\AG \cap \AG'$.
\end{prop}

\begin{proof}
Clear. 
\end{proof}

The Proposition \ref{essentiellement} shows that, 
if the Conjecture A holds for $(W,S)$, there exists 
a unique minimal finite complete rational hyperplane arrangement $\AG$ 
such that the statements (a) and (b) of Conjecture A hold. 
We call the elements of this minimal arrangement 
the {\it essential hyperplanes} of $(W,S)$: 
indeed, if $(W,S)$ is finite, they should be the same as the 
{\it essential hyperplanes} defined by M. Chlouveraki 
\cite[\S 4.3.1]{chlouveraki}, which appear when she studied 
the Rouquier blocks of cyclotomic Hecke algebras associated 
to complex reflection groups. 

Similarly, the Proposition \ref{essentiellement} shows that, 
if the Conjecture A' holds for $(W,S)$, there exists 
a unique minimal finite complete symmetric reduced subset $\EC$ of 
$\ZM[\Sba]$ such that the statements (a) and (b) of Conjecture A' 
holds: they will be called the {\it essential} elements 
of $(W,S)$. 

\bigskip

\remark{essentiel symetrique} 
The Proposition \ref{symetrie} shows that, if the Conjecture A 
holds for $(W,S)$, then the set of essential hyperplanes (respectively 
the set of essential elements) of $(W,S)$ is stable by the action 
of all the symmetries $\t_\o^*$ (respectively $\t_\o$), $\o \in \Sba$.\finl

\bigskip

%
%
%

\remark{parabolique essentiel} 
Let $I$ be a subset of $S$. If $s$, $t \in I$, we shall write 
$s \sim_I t$ if $s$ and $t$ are conjugate in $W_I=\langle I\rangle$. We set $\Iba=I/\!\!\sim_I$. Note that $\Iba$ is not necessarily the image of $I$ 
in $\Sba$. Nevertheless, the inclusion $I \injto S$ induces a map 
$\Iba \to \Sba$, which extends by linearity to a map 
$\g_I : \ZM[\Iba] \to \ZM[\Sba]$. By functoriality (see the Appendix), 
this induces a map $\g_I^* : \positif(\ZM[\Sba]) \to \positif(\ZM[\Sba])$. 
Assume here that Conjecture A holds for $(W,S)$ and $(W_I,I)$: let 
$\AG$ (respectively $\AG_I$) denote the set of essential hyperplanes 
for $(W,S)$ (respectively $(W_I,I)$). Then, since any left cell of $W_I$ 
is the intersection of a left cell of $W$ with $W_I$ 
(see \cite{geck induction}), we get that $(\g_I^*)^{-1}(\AG_I)$ 
is contained in $\AG \cup \{\positif(\ZM[\Sba])\}$. 

Equivalently, if $\EC$ and $\EC_I$ denote the sets of 
essential elements of $(W,S)$ and $(W_I,I)$ respectively, then 
$\g_I(\EC_I) \subseteq \EC \cup \{0\}$.\finl

\bigskip

\example{s=1} 
If $|\Sba|=1$, then Conjecture A is obviously true: the set of essential 
hyperplanes is $\{\HC_S\}$ (note that $\Sba=\{S\}$) and the set of 
essential elements is $\{S,-S\} \subseteq \ZM[\Sba]$. Indeed, 
if $\ph : \Sba \to \G$ and $\ph' : \Sba \to \G'$ are two 
maps such that $\ph(S) > 0$ and $\ph'(S) > 0$, then the relations 
$\sim_?^\ph$ and $\sim_?^{\ph'}$ coincide. On the other hand, if 
$\ph(S)=0$, then $W$ contains only one cell for $\sim_?^\ph$, namely 
$W$ itself: it is clearly the smallest subset of $W$ which is stable 
by translation by $W_\ph=W$.\finl

\bigskip

\subsection{Example: the case where ${\boldsymbol{|\Sba|=2}}$
\label{soussection 2}} 
Since the Conjecture A is expressed in terms of the topology of 
the set $\positif(\ZM[\Sba])$, it might be difficult to understand 
it concretely. The purpose of this example is to give a concrete 
version of this statement whenever $|\Sba|=2$: this will show 
that Conjecture A contains the Conjecture 0 of the introduction, 
and gives some precision for the case where $\ph$ vanishes 
at some simple reflections. 

So assume here that $|\Sba|=2$ and write $\Sba=\{\o_1,\o_2\}$. 
We shall identify $\ZM[\Sba]$ with $\ZM^2$ 
(through $(\l,\mu) \mapsto \l\o_1 + \mu\o_2$). 
If $r$ is a rational number, we shall denote by $\HC_r$ the hyperplane 
$\HC_{-dr,d}$, where $d$ is a non-zero natural number such that $dr \in \ZM$. 
Then $\HC_r=\{\ZM[\Sba],X_r^+,X_r^-\}$, where 
$$X_r^+=\{(\l,\mu) \in \ZM^2~|~\l + r \mu \ge 0\}\quad\text{and}\qquad
X_r^-=\{(\l,\mu) \in \ZM^2~|~\l + r \mu \le 0\}.$$
We set $\HC_\infty=\HC_{(1,0)}$. Then 
$\HC_\infty=\{\ZM[\Sba],X_\infty^+,X_\infty^-\}$, where 
$$X_\infty^+=\{(\l,\mu) \in \ZM^2~|~\mu \ge 0\}\quad\text{and}\qquad
X_\infty^-=\{(\l,\mu) \in \ZM^2~|~\mu \le 0\}.$$

Now, since $\G$ is torsion-free, the natural map 
$\G \to \QM \otimes_\ZM \G$ is injective, so we shall view 
$\G$ as embedded in the $\QM$-vector space $\QM \otimes_\ZM \G$: 
in particular, if $r \in \QM$ and $\g \in \G$, then $r\g$ is well-defined. 
Moreover, the order on $\G$ extends uniquely to a total order on 
$\QM \otimes_\ZM \G$ that we still denote by $\le$. Now, let 
$$\ph(\o_1) = a \quad\text{and}\quad \ph(\o_2)=b.$$
Then, one can immediately translate in concrete terms the fact 
that $\Pos(\ph)$ belongs or not to one of these hyperplanes:

\bigskip

\begin{lem}\label{traduction}
With the above notation, we have:
\begin{itemize}
\itemth{a} $\Pos(\ph) \in \HC_r$ (respectively $\Pos(\ph) \in \HC_\infty$) 
if and only if $b = r a$ (respectively $a=0$).

\itemth{b} $\Pos(\ph) = X_r^+$ (respectively $X_r^-$) if and only if 
$b=ra$ and $a > 0$ (respectively $a < 0$).

\itemth{c} $\Pos(\ph) = X_\infty^+$ (respectively $X_\infty^-$) 
if and only if $a=0$ and $b > 0$ (respectively $b < 0$).
\end{itemize}
\end{lem}

\bigskip

For simplification, we set $\t_i = \t_{\o_i}$. 
Now, let $\AG$ be a finite complete rational hyperplane arrangement 
in $\positif(\ZM[\Sba])$ which is stable under the actions of $\t_1$ 
and $\t_2$. Since $\t_1 \circ \t_2=-\Id_{\ZM[\Sba]}$, this 
is equivalent to say that it is stable under $\t_1$. Since 
all hyperplanes of $\positif(\ZM[\Sba])$ are of the form 
$\HC_r$ for $r \in \QM \cup \{\infty\}$, and since 
$\t_1(\HC_r)=\HC_{-r}$ if $r \in \QM$ (and $\t_1(\HC_\infty)=\HC_\infty$), 
there exist positive rational numbers 
$0 < r_1 < r_2 < \cdots < r_m$ such that 
$$\AG=\{\HC_0,\HC_{r_1},\HC_{-r_1},\HC_{r_2},\HC_{-r_2},\dots,
\HC_{r_m},\HC_{-r_m},\HC_\infty\}.$$
Let us draw $\Pos^{-1}(\AG)$ in 
$\RM[\Sba]^*=\RM \o_1^* \oplus \RM \o_2^*$, where $(\o_1^*,\o_2^*)$ 
is the dual basis of $(\o_1,\o_2)$:
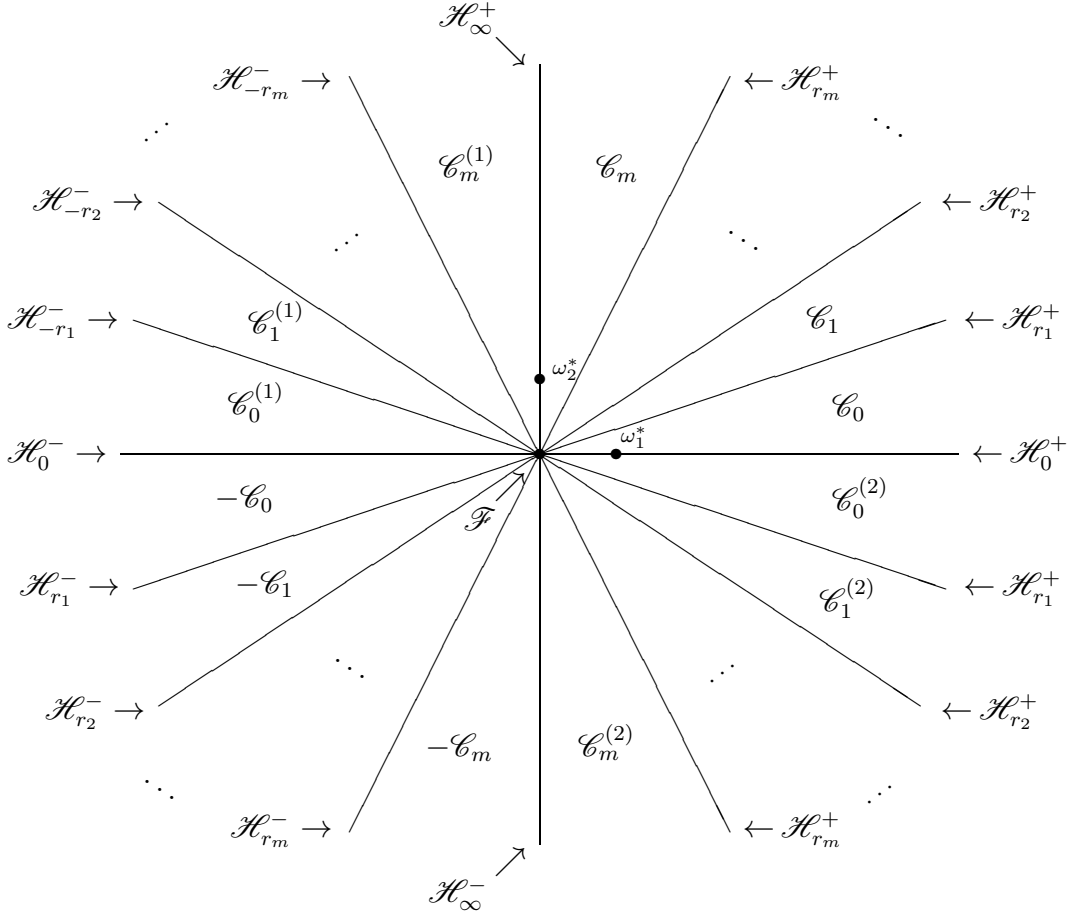
\unitlength0.95pt
\begin{center}
\begin{picture}(450,390)
\put(60,200){\line(1,0){330}}
\put(225,45){\line(0,1){310}}
\put(75,100){\line(3,2){300}}
\put(75,300){\line(3,-2){300}}
\put(150,350){\line(1,-2){150}}
\put(150,50){\line(1,2){150}}
\put(65,146.67){\line(3,1){320}}
\put(65,253.33){\line(3,-1){320}}
\put(225,200){\circle*{4}}
\put(255,200){\circle*{4}}\put(257,205){$\SS{\o_1^*}$}
\put(225,230){\circle*{4}}\put(230,232){$\SS{\o_2^*}$}
\put(15,197){$\HC_0^- \to$}
\put(96,345){$\HC_{-r_m}^- \to$}
\put(103,48){$\HC_{r_m}^- \to$}
\put(20,144){$\HC_{r_1}^- \to$}
\put(15,250){$\HC_{-r_1}^- \to$}
\put(181,22){$\HC_\infty^-$}\put(207,35){$\nearrow$}
\put(185,370){$\HC_\infty^+$}\put(207,357){$\searrow$}
\put(25,297){$\HC_{-r_2}^- \to$}
\put(30,95){$\HC_{r_2}^- \to$}
\put(78,331){\circle*{1.3}}\put(355,62){\circle*{1.3}}
\put(74,328){\circle*{1.3}}\put(359,65){\circle*{1.3}}
\put(70,325){\circle*{1.3}}\put(363,68){\circle*{1.3}}
\put(68,62){$\ddots$}\put(355,325){$\ddots$}
\put(305,345){$\leftarrow$}\put(320,345){$\HC_{r_m}^+$}
\put(383,297){$\leftarrow$}\put(398,297){$\HC_{r_2}^+$}
\put(392,250){$\leftarrow$}\put(407,250){$\HC_{r_1}^+$}
\put(395,197){$\leftarrow$}\put(410,197){$\HC_0^+$}
\put(392,144){$\leftarrow$}\put(407,144){$\HC_{r_1}^+$}
\put(383,95){$\leftarrow$}\put(398,95){$\HC_{r_2}^+$}
\put(305,48){$\leftarrow$}\put(320,48){$\HC_{r_m}^+$}
\put(340,215){$\CC_0$}\put(102,216){$\CC_0^{(1)}$}
\put(340,178){$\CC_0^{(2)}$}\put(97,179){$-\CC_0$}
\put(330,250){$\CC_1$}\put(110,248){$\CC_1^{(1)}$}
\put(335,138){$\CC_1^{(2)}$}\put(105,145){$-\CC_1$}
\put(247,310){$\CC_m$}\put(185,310){$\CC_m^{(1)}$}
\put(240,80){$\CC_m^{(2)}$}\put(180,80){$-\CC_m$}
\put(298,280){$\ddots$}\put(143,110){$\ddots$}
\put(293,110){\circle*{1.3}}
\put(297,113){\circle*{1.3}}
\put(301,116){\circle*{1.3}}
\put(145,281){\circle*{1.3}}
\put(149,284){\circle*{1.3}}
\put(153,287){\circle*{1.3}}
\put(195,170){$\FC$}\put(207,184){$\nearrow$}
\end{picture}
\end{center}
In this figure, we have written $\HC_r^\pm=\{X_r^\pm\}$ and 
we have also drawn all the $\AG$-facets (or at least their 
preimage under $\Pos$): note that 
$\CC_i^{(j)}=\t_j(\CC_i)$ (for $0 \le i \le m$ and 
$1 \le j \le 2$) and that we have the following property 
(where $r_0=0$ and $r_{m+1}=\infty$):

\bigskip

\begin{lem}\label{traduction bis}
Let $i \in \{0,1,2,\dots,m-1,m\}$. 
Then $\Pos(\ph) \in \CC_i$ if and only if 
$r_i a < b < r_{i+1} a$ and $a > 0$.
\end{lem}

\bigskip

Note also that, if $1 \le i \le m$, then $\CC_i$ and $\CC_{i+1}$ 
are the only chambers $\CC$ such that 
$\HC_{r_i}^+ \subseteq \overline{\CC}$, 
that $\CC_0$ and $\CC_0^{(2)}$ are the only chambers $\CC$ such that 
$\HC_0^+ \subseteq \overline{\CC}$ and that 
$\CC_m$ and $\CC_m^{(1)}$ are the only chambers $\CC$ such that 
$\HC_\infty^+ \subseteq \overline{\CC}$. 
All these descriptions (together with Lemma \ref{symetrie}) 
show that, whenever $|\Sba|=2$, Conjecture A is equivalent to the following 
one (the statements (a) and (b) are the same as the statements (a) and (b) 
in the Conjecture 0 of the introduction: the statements (c) and (d) 
gives an extra-information in the case where some parameters 
are equal to zero):

\bigskip

\noindent{\bf Conjecture A''.} 
{\it Assume that $|\Sba|=2$ and let $? \in \{L,R,LR\}$. 
Then there exist rational numbers $0 < r_1 < \cdots < r_m$ 
(depending only on $(W,S)$) such that (setting $r_0=0$ 
and $r_{m+1}=+\infty$), if $\G$ and $\G'$ are two totally 
ordered abelian groups and $\ph : \Sba \to \G$ and 
$\ph' : \Sba \to \G'$ are maps, then:
\begin{itemize}
\itemth{a} If $0 \le i \le m$ and if $\Pos(\ph)$ and $\Pos(\ph')$ 
are in $\CC_i$, then the relations $\sim_?^\ph$ and $\sim_?^{\ph'}$ 
coincide. We denote it by $\sim_?^{\CC_i}$. 

\itemth{b} If $1 \le i \le m$, then $\sim_?^{X_{r_i}^+}$ is generated 
by $\sim_?^{\CC_{i-1}}$ and $\sim_?^{\CC_i}$.

\itemth{c} The relation $\sim_?^{X_0^+}$ is generated by 
$\sim_?^{\CC_0}$ and (left, right or two-sided) translation 
by $W_{\o_2}$.

\itemth{d} The relation $\sim_?^{X_\infty^+}$ is generated by 
$\sim_?^{\CC_m}$ and (left, right or two-sided) translation 
by $W_{\o_1}$.
\end{itemize}}

\bigskip

\section{Conjectures about cell representations} 

\medskip

It seems reasonable to expect that the Conjecture A is compatible 
with the construction of cell representations. We shall make this 
more precise here.

\medskip

\begin{quotation}
\noindent{\bf Hypothesis and notation.} 
{\it In this section, and only in this section, we assume that 
$S$ is finite and that the Conjecture A holds for $(W,S)$. We denote 
by $\AG$ the set of corresponding essential hyperplanes.}
\end{quotation}

\medskip

Let $X \in \positif(\ZM[\Sba])$, let $? \in \{L,R,LR\}$ and let $C$ be a cell 
for $\sim_?^X$. Let $\CC$ be an $\AG$-chamber such that 
$X \in \overline{\CC}$ and let $Y \in \CC$. Since the Conjecture 
A holds, we have $C = \DS{\mathop{\dot{\bigcup}}_{i \in \IC}} C_i$, 
where the $C_i$'s are cells for $\sim_?^Y$. 

\bigskip

\noindent{\bf Conjecture B.} 
{\it There exist a natural number $d$ and a partition 
$\IC = \DS{\mathop{\dot{\bigcup}}_{1 \le k \le d}} \IC_k$ satisfying 
the following properties:
\begin{itemize}
\itemth{a} If $i \in \IC_k$ and $j \in \IC_l$ are such that 
$C_i \pre{?}{Y} C_j$, then $k \le l$.

\itemth{b} There exists a filtration 
$M_0=0 \subseteq M_1 \subseteq \cdots \subseteq M_d=\ZM M_C^X$ 
of the $\ZM W$-$?$-module $\ZM M_C^{?,X}$ such that 
$$M_k/M_{k-1} \simeq \mathop{\oplus}_{i \in \IC_k} \ZM M_{C_i}^{?,Y}.$$
\end{itemize}}

\bigskip

\noindent{\sc Remark - } In the above Conjecture, the ``new'' statement 
is the statement (b). Indeed, the partition of $\IC$ satisfying 
(a) can be obtained by taking the fibers of Lusztig's $\ab$-function, 
if we assume that Lusztig's Conjectures 
\cite[Conjecture 13.4, Conjecture 13.12 (a) and Conjectures 
{\bf P1}-{\bf P15} in Chapter 14]{lusztig} hold 
(indeed, in this case, the $\ab$-function would 
take only finitely many values).\finl

\bigskip

A much weaker version is given by:

\bigskip

\noindent{\bf Conjecture ${\text{\bfseries B}}^{\boldsymbol{-}}$.} 
{\it Assume that $W$ is finite. Then
$$\chi_C^{?,X}=\sum_{i \in \IC} \chi_{C_i}^{?,Y}.$$}

\bigskip

\section{Examples}

\bigskip

We shall illustrate here Conjectures A and B by several examples. 
Most of the computer calcutations that have lead to some of the results of this section 
were done by using M. Geck's programs (using the package {\tt chevie} 
of {\tt GAP3} \cite{chevie}): we thank him warmly for his help. 

\bigskip

\subsection{Finite dihedral groups} 
Assume in this subsection, and only in this subsection, that $S=\{s,t\}$ 
and that $st$ has finite even order $2m$ with $m \ge 2$. So 
we can identify $S$ and $\Sba$. 
Let $w_0=(st)^m=(ts)^m$ be the longest element $W$: it is central.
If $w \in W$, we set $\RC(w)=\{u \in S~|~wu < w\}$. Let
$$\L_s=\{w \in W~|~\RC(w)=\{s\}\}\quad\text{et}\quad
\L_t=\{w \in W~|~\RC(w)=\{t\}\}.$$
An easy computation \cite[\S 8.7]{lusztig} shows that the partition of $W$ 
into left cells for $(W,S,\ph)$ is given by the following table 
(whenever $\ph$ has non-negative values):

\bigskip

$$\begin{array}{|c|c|}
\hline
\espace \ph & \text{Left cells} \\
\hline
\hline
\espace 0=\ph(s)=\ph(t) & W \\
\espace 0=\ph(s) < \ph(t) & \{1,s\},~\L_s\setminus\{s\},~ 
\L_t \setminus\{sw_0\},~ \{sw_0,w_0\} \\
\espace 0 < \ph(s) < \ph(t) & \{1\},~ \{s\},~ \L_s \setminus\{s\},~ 
\L_t \setminus\{sw_0\},~
\{sw_0\},~ \{w_0\} \\
\espace 0 < \ph(s)=\ph(t) & \{1\},~ \L_s,~\L_t,~\{w_0\} \\
\espace 0 < \ph(t) < \ph(s) & \{1\},~ \{t\},~ 
\L_s\setminus\{tw_0\},~ \L_t\setminus\{t\},~ \{tw_0\},~ 
\{w_0\} \\
\espace 0=\ph(t) < \ph(s) & \{1,t\},~ 
\L_s\setminus\{tw_0\},~ \L_t\setminus\{t\},~ \{tw_0,w_0\}\\
\hline
\end{array}$$

\bigskip

\begin{prop}\label{diedral}
The Conjectures A and B hold if $|S|=2$ and $|W| < \infty$. The essential 
hyperplanes are $\HC_s$, $\HC_t$, $\HC_{s-t}$ and 
$\HC_{s+t}$. 
\end{prop}

\begin{proof}
Let us first show Conjecture A. By the discussion of 
\S\ref{soussection 2}, it is sufficient to show Conjecture A''. 
For this, take $m=1$ and $r_1=1$: then the statements (a), (b), (c) and 
(d) are easily checked by inspection of the above table.

\medskip

Let us now show Conjecture B. There are finitely many cases to be 
considered: using the automorphism of the Coxeter graph of $(W,S)$, 
the number of cases can be drastically limited. We shall only consider 
the following one, the other ones being treated similarly. 
Let $\ph : \Sba \to \G$, $\ph' : \Sba \to \G'$ 
and $\ph'' : \Sba \to \G''$ be maps such that 
$$\ph(s)=\ph(t) > 0,\quad\ph'(s) > \ph'(t) > 0\quad
\text{and}\quad \ph''(t) > \ph''(s) > 0.$$
Here, $\G$, $\G'$ and $\G''$ are totally ordered abelian groups. 
We also denote by $\aug : \HC(W,S,\ph) \to \ZM W$, 
$\aug' : \HC(W,S,\ph') \to \ZM W$ and $\aug'' : \HC(W,S,\ph'') \to \ZM W$ 
the morphisms of rings induced respectively by the augmentation morphisms 
$\ZM[\G] \to \ZM$, $\ZM[\G'] \to \ZM$ and $\ZM[\G''] \to \ZM$. If $w \in W$, 
we set
$$G_w=\aug(C_w^\ph),\quad G_w'=\aug'(C_w^{\ph'})\quad\text{and}\quad
G_w''=\aug''(C_w^{\ph''}).$$
Let $C=\L_s$: it is a left cell for $(W,S,\ph)$ and,
by the Table above (and easy computations), we have that:
\begin{itemize}
\itemth{\CC'}\quad $C\setminus\{tw_0\}$ and $\{tw_0\}$ are left cells for 
$(W,S,\ph')$ and $\{tw_0\} <_L^{\ph'} C\setminus\{tw_0\}$.

\itemth{\CC''}\quad $\{s\}$ and $C\setminus\{s\}$ are left cells for 
$(W,S,\ph'')$ and $C\setminus\{s\} <_L^{\ph''} \{s\}$.
\end{itemize}
Note that $G_{w_0}=G_{w_0}'=G_{w_0}''=\sum_{w \in W} w$. We set
$$I=\mathop{\oplus}_{\substack{w \in W\\ ws < w}} \ZM G_w,\quad
I'=\mathop{\oplus}_{\substack{w \in W\\ ws < w}} \ZM G_w'\quad \text{and}\quad 
I''=\mathop{\oplus}_{\substack{w \in W\\ ws < w}} \ZM G_w''.$$
Then $I$, $I'$ and $I''$ are left ideals of $\ZM W$ 
(by \cite[Lemma 8.6]{lusztig}) and 
$$I=I'=I''.$$
Indeed, this last equality can be proved by tedious computations 
using formulas given by Lusztig \cite[Propositions 7.3 and 7.6]{lusztig}, but 
it might also be proved by using \cite[Lemma 8.4]{lusztig}: this 
last lemma shows that, since $G_s=G_s'=G_s''=s+1$, 
$$I=I'=I''=\{h \in \ZM W~|~h (s+1)=2h\}=\{h \in \ZM W~|~hs=h\}.$$
Now, by definition,
$$\ZM M_C^{L,\ph} \simeq I/\ZM G_{w_0}.$$
Now, let 
$$M'=\ZM G_{tw_0}' \oplus \ZM G_{w_0}'\quad
\text{and}\quad M''=\mathop{\oplus}_{w \in (C \cup \{w_0\})\setminus \{s\}} 
\ZM G_w''.$$
Then 
$$M_0'=0 \subseteq M_1'=M'/\ZM G_{w_0} \subseteq 
M_2'=I'/\ZM G_{w_0}=\ZM M_C^{L,\ph}$$
$$M_0''=0 \subseteq M_1''=M''/\ZM G_{w_0} \subseteq M_2''= 
I''/\ZM G_{w_0}=\ZM M_C^{L,\ph}\leqno{\text{and}}$$
are two filtrations of $\ZM M_C^{L,\ph}$ by left $\ZM W$-modules. Moreover, 
by definition, 
\begin{itemize}
\item $M_1'/M_0' \simeq \ZM M_{\{tw_0\}}^{L,\ph'}$ and 
$M_2'/M_1' \simeq \ZM M_{C \setminus\{tw_0\}}^{L,\ph'}$.

\item $M_1''/M_0'' \simeq \ZM M_{C \setminus \{s\}}^{L,\ph''}$ and 
$M_2''/M_1'' \simeq \ZM M_{\{s\}}^{L,\ph''}$.
\end{itemize}
This shows Conjecture B in this particular case, taking into account 
the statements $(\CC')$ and $(\CC'')$ above. 
\end{proof}

\bigskip

\subsection{Infinite dihedral group}
The same argument as in the finite case (see Proposition \ref{diedral}) 
shows that the following proposition holds:

\bigskip

\begin{prop}\label{diedral infini}
If $S=\{s,t\}$ and $st$ has infinite order, then the Conjectures 
A and B hold. The essential hyperplanes of $(W,S)$ are $\HC_s$, 
$\HC_t$, $\HC_{s-t}$ and $\HC_{s+t}$. 
\end{prop}

\bigskip

For general $(W,S)$, the Remark \ref{parabolique essentiel} 
together with the Propositions \ref{diedral} and \ref{diedral infini} 
shows immediately that:

\bigskip

\begin{prop}\label{s-t}
Assume that Conjecture A holds for $(W,S)$. 
Let $s$ and $t$ be two elements of $S$ such that $\sba \neq \tba$ 
and such that the order of $st$ is greater than $3$. Then 
$\HC_{\sba-\tba}$ and $\HC_{\sba+\tba}$ are essential hyperplanes 
of $(W,S)$. Equivalently, $\sba-\tba$ and $\sba+\tba$ are essential 
elements of $(W,S)$. 
\end{prop}

\bigskip

\subsection{Type ${\boldsymbol{F_4}}$}
Assume in this subsection, and only in this subsection, 
that $(W,S)$ is of type $F_4$. Let $S=\{s_1,s_2,t_1,t_2\}$ be 
indexed in such a way that the Coxeter graph of $(W,S)$ is
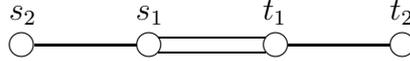
\begin{center}
\begin{picture}(160,28)
\put(0,10){\circle{10}}\put(-5,20){${s_2}$}
\put(50,10){\circle{10}}\put(45,20){${s_1}$}
\put(100,10){\circle{10}}\put(95,20){${t_1}$}
\put(150,10){\circle{10}}\put(145,20){${t_2}$}
\put(5,10){\line(1,0){40}}
\put(54,13){\line(1,0){42}}
\put(54,7){\line(1,0){42}}
\put(105,10){\line(1,0){40}}
\end{picture}
\end{center}
Let $s=\{s_1,s_2\}$ and $t=\{t_1,t_2\}$, so that $\Sba=\{s,t\}$. 

\bigskip

\begin{theo}[Geck]
If $(W,S)$ is of type $F_4$, then the Conjectures A and ${\text{B}}^-$ hold 
for $(W,S)$. With the above notation, the essential hyperplanes of 
$(W,S)$ are $\HC_s$, $\HC_t$, $\HC_{s-2t}$, $\HC_{s-t}$, $\HC_{2s-t}$, 
$\HC_{s+2t}$, $\HC_{s+t}$ and $\HC_{2s+t}$.
\end{theo}

\bigskip

\noindent{\sc Remark - } With the notation of \S\ref{soussection 2}, 
this is equivalent to say that Conjecture A'' holds and that 
we can take $m=3$, $r_1=1/2$, $r_2=1$ and $r_3=2$.\finl

%
\bigskip

\begin{proof}
In \cite{geck f4}, M. Geck has computed the Kazhdan-Lusztig cells $W$ for all 
choices of map $\ph$ so that $\ph(S) \subset \G_{>0}$: this involves 
both theoretical and computational (using {\tt GAP3}) arguments. 
To get the cells whenever $\ph(S) \subset \G_{\geqslant 0}$, 
it is then sufficient to use Corollary \ref{KL semi-direct} 
(and the decomposition $W=\SG_3 \ltimes W(D_4)$, see \ref{dyer bis}) 
together with the knowledge of the cells in type $D_4$ (which 
can be obtained again by using {\tt GAP3}). Using these results, 
one can easily check Conjecture A'' (so that Conjecture A holds). 
Conjecture $\Brm^-$ is then also checked by computations using {\tt GAP3}. 
Note that all these computations are simplified thanks to the involutive 
automorphism of $(W,S)$ that sends $s_i$ to $t_i$.
\end{proof}

\bigskip

\subsection{Type ${\boldsymbol{B}}$\label{section B}}
Assume in this subsection, and only in this subsection, that 
$(W,S)$ is the pair 
$(W_n,S_n)$, where $W_n$ is of type $B_n$ ($n \ge 2$), that  $S_n=\{t,s_1,s_2,\dots,s_{n-1}\}$, and that the Coxeter graph of 
$(W_n,S_n)$ is
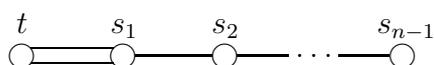
\begin{center}
\begin{picture}(170,40)
\put(10,10){\circle{10}}\put(8,19){$t$}
\put(50,10){\circle{10}}\put(45,19){$s_1$}
\put(90,10){\circle{10}}\put(85,19){$s_2$}
\put(160,10){\circle{10}}\put(150,19){$s_{n-1}$}
\put(13.9,13){\line(1,0){32.3}}
\put(13.7,7){\line(1,0){32.5}}
\put(55,10){\line(1,0){30}}
\put(95,10){\line(1,0){20}}
\put(155,10){\line(-1,0){20}}
\put(118,9.4){$\dots$}
\end{picture}
\end{center}
Identify $t$ and $\{t\}$ and let $s=\{s_1,s_2,\dots,s_{n-1}\}$. 
So $\Sba=\{s,t\}$. In this case, the Conjectures of Geck, Iancu, 
Lam and the author \cite[Conjectures A et B]{bgil} suggest the following 
one:

\bigskip

\noindent{\bf Conjecture C.} 
{\it The Conjectures A and B are true for $(W_n,S_n)$. The essential hyperplanes are $\HC_s$, $\HC_t$, $\HC_{t-is}$ ($1 \le i \le n-1$) and 
$\HC_{t+is}$ ($1 \le i \le n-1$).}

\bigskip

\noindent{\sc Comment - } The reader may refer to the papers 
\cite{lacriced}, \cite{bonnafe two}, \cite{bgil}, \cite{bonnafe B} et \cite{bonnafe guilhot} (in chronological order) for results, 
comments, evidences for this conjecture. Note also the work of 
Pietraho \cite{pietraho} in relation with Conjecture B.\finl

\bigskip

%
%

\subsection{Type ${\boldsymbol{\widetilde{G}_2}}$}
Assume in this subsection, and only in this subsection, that 
$S=\{t,s_1,s_2\}$, that $(W,S)$ is a Coxeter system of type 
$\widetilde{G}_2$, and that the Coxeter graph of $(W,S)$ is given by:
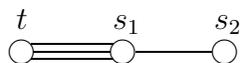
\begin{center}
\begin{picture}(100,30)
\put(10,10){\circle{10}}\put(8,19){$t$}
\put(50,10){\circle{10}}\put(46,19){$s_1$}
\put(90,10){\circle{10}}\put(86,19){$s_2$}
\put(55,10){\line(1,0){30}}
\put(15,10){\line(1,0){30}}
\put(14,13){\line(1,0){32}}
\put(14,7){\line(1,0){32}}
\end{picture}
\end{center}

\noindent Let $s=\{s_1,s_2\}$ and let us identify $t$ and $\{t\}$, so that $\Sba=\{s,t\}$. 
J. Guilhot proved that Conjecture A holds in this case (see \cite{guilhot} and \cite{guilhot2}):

\bigskip

\begin{theo}[Guilhot]
If $(W,S)$ is an affine Weyl group of type $\widetilde{G}_2$, then 
the Conjecture A holds. With the above notation, the essential 
hyperplanes are $\HC_s$, $\HC_t$, $\HC_{s-t}$, $\HC_{s+t}$, 
$\HC_{2t-3s}$, $\HC_{2t+3s}$, $\HC_{t-2s}$ and $\HC_{t+2s}$.
\end{theo}

\bigskip

\subsection{Type ${\boldsymbol{\widetilde{B}_2}}$}
Assume in this subsection, and only in this subsection, that 
$S=\{t,t',s\}$, that $(W,S)$ is a Coxeter system of type 
$\widetilde{B}_2$, and that the Coxeter graph of $(W,S)$ is given by:
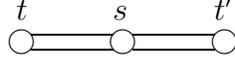
\begin{center}
\begin{picture}(100,30)
\put(10,10){\circle{10}}\put(8,19){$t$}
\put(50,10){\circle{10}}\put(46,19){$s$}
\put(90,10){\circle{10}}\put(86,19){$t'$}
\put(54,13){\line(1,0){32}}
\put(54,7){\line(1,0){32}}
\put(14,13){\line(1,0){32}}
\put(14,7){\line(1,0){32}}
\end{picture}
\end{center}

\noindent 
Let us identify $t$, $t'$ and $s$ with $\{t\}$, $\{t'\}$ and $\{s\}$ respectively, 
so that $\Sba=\{s,t,t'\}$. 
J. Guilhot proved that Conjecture A holds in this case (see \cite{guilhot} and \cite{guilhot2}):

\bigskip

\begin{theo}[Guilhot]
If $(W,S)$ is an affine Weyl group of type $\widetilde{B}_2$, then 
the Conjecture A holds. With the above notation, the essential 
hyperplanes are $\HC_s$, $\HC_t$, $\HC_{t'}$, $\HC_{s-t}$, $\HC_{s+t}$, 
$\HC_{s-t'}$, $\HC_{s+t'}$, $\HC_{t-t'}$, $\HC_{t+t'}$, $\HC_{s+t+t'}$, 
$\HC_{s+t-t'}$, $\HC_{s-t+t'}$, $\HC_{s-t-t'}$, $\HC_{2s+t+t'}$, 
$\HC_{2s+t-t'}$, $\HC_{2s-t+t'}$ and $\HC_{2s-t-t'}$.
\end{theo}

\bigskip

This is the only case where $|\Sba|=3$ and where the decomposition into cells have 
been obtained for all choices of parameters.

\bigskip

\subsection{General results about finite Coxeter groups}
Though it is stated in a different context, 
the next result is proved by Geck 
in \cite[\S 3]{geck f4}.

\bigskip

\begin{prop}[Geck]\label{finitude}
If $W$ is finite, then the Conjecture A (a) holds. 
\end{prop}

%
%

%
%
%
\newpage

\appendix

\section{Positive subsets of an abelian group\label{section positive}}

\medskip

\begin{quotation}
{\it In this appendix, we fix a free $\ZM$-module $\L$ of finite rank 
and we set $V=\RM \otimes_\ZM \L$.}
\end{quotation}

\medskip

The aim of this Appendix is to gather some properties of 
{\it positive subsets} of $\L$ (as defined in \cite[\S 1]{bonnafe positif}). 
We recall how the set of such positive subsets can be endowed 
with a topology, and some particular features of this topology.

\def\tors{{\mathrm{tor}}}

\bigskip

\noindent{\bf Definitions, preliminaries.\label{sous positive}} 
A subset $X$ of $\L$ is called {\it positive} if the following three 
conditions are fulfilled:
\begin{itemize}\itemindent1cm
\itemth{P1} $\L=X \cup (-X)$.

\itemth{P2} $X+X \subset X$.

\itemth{P3} $X \cap (-X)$ is a subgroup of $\L$.
\end{itemize}
Let $\positif(\L)$ denote the set of positive subsets of $\L$. If 
$\G$ is a totally ordered group and if $\ph : \L \to \G$ is a morphism 
of groups, we set
$$\Pos(\ph)=\{\l \in \L~|~\ph(\l) \ge 0\}.$$
It is then clear that 
\equatap\label{noyau pos}
\Ker \ph = \Pos(\ph) \cap \Pos(-\ph) = \Pos(\ph) \cap -\Pos(\ph)
\endequatap
and that
\equatap\label{positif exemple}
\text{\it $\Pos(\ph)$ is a positive subset of $\L$}.
\endequatap
%

%
%
%
%

The converse of property \ref{positif exemple} holds. Indeed, 
let $X \in \positif(\L)$ and let $\can_X : \L \to \L/(X \cap(-X))$ 
denote the canonical morphism; if $\g$ and $\g'$ 
belong to $\L/(X \cap(-X))$, we shall write $\g \leqslant_X \g'$ 
if there exists a representative of $\g' - \g$ which belongs to $X$. 
It is easily seen that 
\equatap\label{relation independante}
\text{\it $\g \leqslant_X \g'$ if and only if all the representatives 
of $\g'-\g$ belong to $X$.}
\endequatap
We then deduce immediately from properties (P1), (P2) and (P3) 
of positive subsets that
\equatap\label{ordre X}
\text{\it $(\L/(X \cap (-X)),\le_X)$ is a totally ordered abelian group}
\endequatap
and that 
\equatap\label{X pos}
X=\Pos(\can_X).
\endequatap

\bigskip

\noindent{\bf Functoriality.} 
If $\s : \L \to \L'$ is a morphism between two free abelian groups 
of finite rank, then \cite[Lemma 1.2]{bonnafe positif} the map 
$$\fonction{\s^*}{\positif(\L')}{\positif(\L)}{X}{\s^{-1}(X)}$$
is well-defined. This implies that $\positif$ is a contravariant 
functor from the category of free $\ZM$-modules of finite rank 
to the category of sets \cite[\S 1.2]{bonnafe positif}.

\bigskip

\noindent{\bf Linear forms.} 
Let $V^*$ denote the dual of $V$ (i.e. $V^*=\Hom_\RM(V,\RM)$). 
If $\ph \in V^*$, we shall denote for simplification 
$\Pos(\ph)$ the positive subset $\Pos(\ph|_\L)$. This defines a map 
$$\Pos : V^* \longto \positif(\L).$$
We recall how to classify the positive subsets of $\L$ 
thanks to this application. Let $\FC(\L)$ denote the set 
of finite sequences $(\ph_1,\dots,\ph_r)$ such that, for all 
$i \in \{1,2,\dots,r\}$, $\ph_i$ is a non-zero linear form on 
$\RM \otimes_\ZM (\L \cap \Ker \ph_{i-1})$ (with the convention that 
$\ph_0=0$). By convention, we assume that the empty sequence, 
denoted by $\vide$, belongs to $\FC(\L)$. 

Let $d = \dim V$. 
Note that, if $(\ph_1,\dots,\ph_r) \in \FC(\L)$, then $r \le d$. 
We define the following action of $(\RM_{>0})^d$ on $\FC(\L)$: 
if $(\kappa_1,\dots,\kappa_d) \in (\RM_{>0})^d$ and if 
$(\ph_1,\dots,\ph_r) \in \FC(\L)$, we set 
$$(\kappa_1,\dots,\kappa_d)\cdot (\ph_1,\dots,\ph_r) = 
(\kappa_1 \ph_1,\dots,\kappa_r \ph_r).$$
Let us endow $\RM^r$ with the lexicographic order: then $\RM^r$ is a totally 
ordered abelian group and $(\ph_1,\dots,\ph_r) : \L \to \RM^r$ 
is a morphism of groups. So $\Pos(\ph_1,\dots,\ph_r)$ is well-defined 
and belong to $\positif(\L)$ (see \ref{positif exemple}). In fact, all the 
positive subsets of $\L$ are obtained in this way
\cite[Proposition 1.10]{bonnafe positif}:

\bigskip

\begin{propap}\label{surjection pos}
The map
$$\fonctio{\FC(\L)}{\positif(\L)}{\phb}{\Pos(\phb)}$$
is well-defined and induces a bijection
$\FC(\L)/(\RM_{> 0})^d \stackrel{\sim}{\longrightarrow} \positif(\L)$. 
\end{propap}

\noindent{\sc Remark - } By convention, $\Pos(\vide)=\L$.\finl

\bigskip

\noindent{\bf Topology on 
${\boldsymbol{\positif(\L)}}$.} 
If $E$ is a subset of $\L$, we set 
$$\UC(E)=\{X \in \positif(\L)~|~X \cap E=\vide\}.$$
If $\l_1$,\dots, $\l_n$ are elements of $\L$, we shall denote for 
simplification $\UC(\l_1,\dots,\l_n)$ the set 
$\UC(\{\l_1,\dots,\l_n\})$. If this is necessary, these sets will be denoted by 
$\UC_\L(E)$ or $\UC_\L(\l_1,\dots,\l_n)$. Then
\equatap\label{u inter}
\UC(E)=\bigcap_{\l \in E} \UC(\l).
\endequatap
Note that 
\equatap\label{u vide}
\UC(\vide)=\positif(\L)\qquad\text{and}\qquad \UC(\L)=\{\vide\}.
\endequatap
On the other hand, if $(E_i)_{i \in I}$ is a family of subsets of $\L$, then 
\equatap\label{u intersection}
\bigcap_{i \in I} \UC(E_i) = \UC\bigl(\bigcup_{i \in I} E_i\bigr).
\endequatap
A subset $\UC$ of $\positif(\L)$ is called {\it open} if, 
for all $X \in \UC$, there exists a {\bf finite} subset $E$ of $\L$ such that 
$X \in \UC(E)$ and $\UC(E) \subset \UC$ (see \cite[\S 2.1]{bonnafe positif}). 
The equality \ref{u intersection} shows that this defines a topology on 
$\positif(\L)$. 

\bigskip

\exampleap{Z} 
The topological space $\positif(\ZM)$ has only three points: 
$\ZM$, $\ZM_{\ge 0}$ and $\ZM_{\le 0}$. Among them, 
only $\ZM$ is closed whereas $\ZM_{\ge 0}$ and $\ZM_{\le 0}$ 
are open (indeed, $\{\ZM_{\ge 0}\} = \UC(-1)$ 
and $\{\ZM_{\le 0}\} = \UC(1)$).\finl

\bigskip

The following result has been shown in 
\cite[Proposition 2.7]{bonnafe positif}:

\bigskip

\begin{theoap}\label{continu}
The map $\Pos : V^* \longto \positif(\L)$ is continuous and dominant 
(i.e. its image is dense). Moreover, it induces an homeomorphism between its 
image and $V^*/\RM_{>0}$. 
\end{theoap}

\bigskip

Recall that, if $\s : \L \to \L'$ is a morphism between two free abelian 
groups of finite rank, 
then \cite[Proposition 2.6]{bonnafe positif} the map 
$\s^* : \positif(\L') \to \positif(\L)$ is continuous. So 
$\positif$ is a contravariant functor from the category of free abelian groups 
of finite rank to the category of topological spaces.

\bigskip

\noindent{\bf Rational subspaces.} 
If $E$ is a subset of $\L$, we set 
$$\LC(E)=\{X \in \positif(\L)~|~E \subset X \cap (-X)\}.$$
Note that 
$$\Pos^{-1}(\LC(E))=\{\ph \in V^*~|~\forall~\l \in E,~\ph(\l)=0\}=E^\perp.$$
A {\it rational subspace} of $\positif(\L)$ is a subset of $\positif(\L)$ 
of the form $\LC(E)$, where $E$ is a subset of $\L$. If 
$\l \in \L\setminus \{0\}$, we shall denote by $\HC_\l$ the rational 
subspace $\LC(\{\l\})$: such a rational subspace is called a 
{\it rational hyperplane}. Note that 
\equatap\label{partage espace}
\positif(\L)= \UC(\l) \dotcup \HC_\l \dotcup \UC(-\l).
\endequatap

\bigskip

%
%
%
%
%
%
%

\noindent{\bf Half-spaces.} 
Let $\HC$ be a rational hyperplane of $\positif(\L)$ and let 
$\l\in \L\setminus\{0\}$ be such that $\HC=\HC_\l$. By \ref{partage espace}, 
the hyperplane $\HC$ defines a unique equivalence relation 
$\smile_{\!\HC}$ on $\positif(\L)$ such that the equivalence classes are 
$\UC(\l)$, $\HC$ and $\UC(-\l)$: note that this relation does not depend 
on the choice of $\l$. Moreover \cite[Proposition 3.2]{bonnafe positif}:

\bigskip

\begin{propap}\label{composantes connexes}
$\HC$ is a closed subset of $\positif(\L)$ and $\UC(\l)$ and $\UC(-\l)$ 
are the connected components of $\positif(\L)\setminus \HC$. Moreover 
$$\overline{\UC(\l)}=\UC(\l) \cup \HC_\l.$$
\end{propap}

\bigskip

\noindent{\bf Hyperplane arrangements.} 
From now on, and until the end of this Appendix, we shall work under 
the following hypothesis:

\bigskip

\begin{quotation}
{\it We fix a {\bfit finite} set $\AG$ of rational hyperplanes 
of $\positif(\L)$.}
\end{quotation}

\bigskip

Following \cite[\S 3.3]{bonnafe positif}, we shall recall the notions 
of {\it facets}, {\it chambers} and {\it faces} associated to $\AG$, 
in a similar way as these notions are defined for ``real'' hyperplane 
arrangements \cite[Chapitre V, \S 1]{bourbaki}. 

Let $\smile_\AG$ denote the equivalence relation on $\positif(\L)$ defined 
as follows: if $X$ and $Y$ are two elements of $\positif(\L)$, we shall write 
$X \smile_\AG Y$ if $X \smile_{\!\HC} Y$ for all $\HC \in \AG$. 
We shall call {\it facets} (or {\it $\AG$-facets}) the equivalence classes 
for the relation $\smile_\AG$. We shall call {\it chambers} (or
{\it $\AG$-chambers}) the facets which meet no hyperplane of $\AG$. 
If $\FC$ is a facet, we set 
$$\LC_{\AG} ( \FC) = \bigcap_{\overset{\HC \in \AG}{\FC \subset \HC}} \HC,$$
with the usual convention that $\LC_{\AG} (\FC) = \positif (\L)$ if 
$\FC$ is a chamber. This rational subspace will be called the {\it support} 
of $\FC$ and we define the {\it dimension} of $\FC$ to be the 
non-negative integer 
$$\dim \FC = \dim_{\RM} \Pos^{- 1} ( \LC_{\AG} (\FC)) .$$
Similarly, we shall call the {\it codimension} of $\FC$ the 
non-negative integer 
$$\codim \FC=\dim_\RM V - \dim \FC.$$
With these definitions, a chamber is a facet of codimension $0$. The next 
proposition has been proved in \cite[Proposition 3.3]{bonnafe positif}. 
The set of $\AG$-facets (respectively $\AG$-chambers) is denoted 
by $\facets(\AG)$ (respectively $\chambres(\AG)$). 

\bigskip

\begin{propap}\label{facettes}
Let $\FC$ be a facet and let $X \in \FC$. Then:
\begin{itemize}
\itemth{a} $\FC=\DS{\bigcap_{\HC \in \AG}} \DC_\HC(X)$.

\itemth{b} $\overline{\FC}=\DS{\bigcap_{\HC \in \AG}} \overline{\DC_\HC(X)}$.

\itemth{c} $\overline{\FC}$ is the union of $\FC$ and of facets 
of strictly smaller dimension. 

\itemth{d} If $\FC'$ is a facet such that $\overline{\FC}=\overline{\FC}'$, 
then $\FC=\FC'$.
\end{itemize}
\end{propap}

\bigskip

\bigskip

We define a relation $\infspe$ on the set of facets: we write 
$\FC \infspe \FC'$ if $\overline{\FC} \subseteq \overline{\FC}'$ 
(i.e. if $\FC \subseteq \overline{\FC}'$). 
The Proposition \ref{facettes} (d) shows that:

\bigskip

\begin{coroap}\label{ordre facettes}
The relation $\infspe$ is a partial order. 
\end{coroap}

\bigskip

Recall that a map $\xi : \XC \to P$, where $\XC$ is a topological space and 
$P$ is a partially ordered set, is called {\it upper semicontinuous} if, 
for all $p \in P$, the set $\{x \in \XC~|~\xi(x) < p\}$ is open. 

\bigskip

\begin{coroap}\label{upper}
Let $P$ be a partially ordered set and let $\xi : \positif(\L) \to P$ 
be a map satisfying the following properties:
\begin{itemize}
\itemth{1} $\xi$ is constant on facets (if $\FC$ is a facet, we denote 
by $\xi(\FC)$ the value of $\xi$ on $\FC$). 

\itemth{2} If $\FC$ and $\FC'$ are two facets such that 
$\FC \infspe \FC'$, then $\xi(\FC) \ge \xi(\FC')$.
\end{itemize}
Then $\xi$ is upper semicontinuous.
\end{coroap}

\begin{proof}
Let $p \in P$ and let $\UC=\{x \in \positif(\L)~|~\xi(x) < p\}$. 
Let $x \in \UC$ and let $\AG_x=\{\HC \in \AG~|~x \not\in \HC\}$. 
Let $\CC$ be the $\AG_x$-facet containing $x$. Then 
$\CC$ is an $\AG_x$-chamber (because $x$ does not belong to any 
hyperplane in $\AG_x$), and so $\CC$ is open. To prove the corollary, 
it is enough to show that $\CC \subseteq \UC$. 

For this, let $y \in \CC$ and let $\FC'$ be the $\AG$-facet containing 
$y$. Then $\FC' \subseteq \CC$ and so it is sufficient 
to show that $\FC' \subseteq \UC$. Let $\FC$ denote the $\AG$-facet 
of $x$: by the properties (1) and (2), we only need to show that 
$\FC \infspe \FC'$. 

Then, by Corollary \ref{facettes}, 
$$\overline{\FC}'=\Bigl(\bigcap_{\HC \in \AG_x} \overline{\DC_\HC(y)}\Bigr) ~\bigcap~ 
\Bigl(\bigcap_{\HC \in \AG\setminus \AG_x} \overline{\DC_\HC(y)}\Bigr)$$
$$\overline{\FC}=\Bigl(\bigcap_{\HC \in \AG_x} \DC_\HC(x)\Bigr) ~\bigcap~ 
\Bigl(\bigcap_{\HC \in \AG\setminus \AG_x} \overline{\DC_\HC(x)}\Bigr).
\leqno{\text{and}}$$
But, since $x$ and $y$ belongs to $\CC$, we have, again by Corollary 
\ref{facettes} (but applied to $\AG_x$), 
$$\bigcap_{\HC \in \AG\setminus \AG_x} \overline{\DC_\HC(y)}=
\bigcap_{\HC \in \AG\setminus \AG_x} \overline{\DC_\HC(x)}=\overline{\CC}.$$
So
$$\overline{\FC}' = \overline{\CC} ~\bigcap~ 
\Bigl(\bigcap_{\HC \in \AG\setminus \AG_x} \overline{\DC_\HC(y)}\Bigr)$$
$$\overline{\FC}=\overline{\CC} ~\bigcap~ 
\Bigl(\bigcap_{\HC \in \AG\setminus \AG_x} \HC\Bigr).
\leqno{\text{and}}$$
Now, if $\HC \in \AG_x$, then $\overline{\DC_\HC(y)}$ contains 
$\HC$, so $\overline{\FC} \subseteq \overline{\FC}'$, as expected.
\end{proof}

\bigskip

\noindent{\sc Remark - } Under the hypothesis of the Proposition, 
the above proof can be followed word by word to show that the set 
$\{x \in \positif(\L)~|~\xi(x) \le p\}$ is also open.\finl


\bigskip














\bigskip


\begin{thebibliography}{131}

%

\bibitem{bonnafe two} {\sc C. Bonnaf\'e}, 
Two-sided cells in type $B$ (asymptotic case), {\it J. Algebra} {\bf 304} 
(2006), 216-236.

\smallskip

\bibitem{bonnafe B} {\sc C. Bonnaf\'e}, 
On Kazhdan-Lusztig cells in type $B$, {\it Journal of Alg. Combinatorics}, 
{\bf 31} (2010), 53-82.

\smallskip

\bibitem{bonnafe dyer} {\sc C. Bonnaf\'e \& M. Dyer}, 
Semidirect product decomposition of Coxeter groups, preprint (2008), 
to appear in {\it Comm. in Algebra}, 
available at {\tt arXiv:0805.4100}. 

\smallskip

\bibitem{bonnafe positif} {\sc C. Bonnaf\'e}, Topologie 
sur l'ensemble des parties positives d'un r\'eseau, 
{\it Ann. Sc. Math. Qu\'ebec} {\bf 32} (2008), 11-27.

\smallskip

\bibitem{bgil} {\sc C. Bonnaf\'e, M. Geck, L. Iancu \& T. Lam}, 
On domino insertion and Kazhdan-Lusztig cells in type $B_n$, 
to appear in {\it Progress in Mathematics} (Birkh\"auser), 
available at {\tt arXiv:math/0609279}.

\smallskip

\bibitem{bonnafe guilhot} {\sc C. Bonnaf\'e \& J. Guilhot}, 
On Kazhdan-Lusztig cells in the asymptotic case, {\it in preparation}.

\smallskip

\bibitem{lacriced} {\sc C. Bonnaf\'e \& L. Iancu}, Left cells in type $B_n$ 
with unequal parameters, {\it Represent. Theory} {\bf 7}  (2003), 587--609.

\smallskip

\bibitem{bourbaki} {\sc N. Bourbaki}, {\it Groupes et alg\`ebres de Lie, 
chapitres IV, V et VI}, Hermann, Paris, 1968.

\smallskip


\bibitem{chlouveraki} {\sc M. Chlouveraki}, 
{\it Blocks and families for Cyclotomic Hecke algebras}, 
Lecture Notes in Mathematics {\bf 1981}, 2009, Springer.

\smallskip

%
%
%

\bibitem{geck induction} {\sc M.~Geck}, 
On the induction of Kazhdan-Lusztig cells, 
{\it Bull. London Math. Soc.} {\bf 35} (2003), 608-614. 

\smallskip

\bibitem{geck f4} {\sc M. Geck}, 
Computing Kazhdan-Lusztig cells for unequal parameters, 
{\it J. Algebra} {\bf 281} (2004), 342--365.

\smallskip

\bibitem{chevie} {\sc M. Geck, G. Hi\ss, F. L\"ubeck, G. Malle \& G. Pfeiffer}, CHEVIE - A system for computing and processing
generic character tables, {\it Appl. Algebra Engrg. Comm. Comput.} 
{\bf 7} (1996) 175-210; 
available at {\tt http://www.math.rwth-aachen.de/${\SS{\sim}}$CHEVIE}.


%



\bibitem{guilhot2} {\sc J. Guilhot}, Computations of Kazhdan-Lusztig cells 
in type $\widetilde{G}_2$: the unequal parameter case, Ph.D. Thesis (2008), 
Aberdeen and Lyon.

\smallskip

\bibitem{guilhot} {\sc J. Guilhot}, Generalized induction of 
Kazhdan-Lusztig cells, {\it Ann. Inst. Four.} {\bf 59} (2009), 1349--1376.

\smallskip

\bibitem{guilhot rang 2} {\sc J. Guilhot}, Kazhdan-Lusztig cells in the affine Weyl groups of rank $2$,
preprint (2009), available at {\tt arXiv:0901.1711}.

\bibitem{KaLu}
{\sc D.~A. Kazhdan and G. Lusztig}, Representations of {C}oxeter groups and
Hecke algebras, {\it Invent. Math.} \textbf{53} (1979), 165--184.

\smallskip

\bibitem{gal} {\sc S. Gal}, 
On normal subgroups of Coxeter groups generated by standard 
parabolic subgroups, {\it Geom. Ded.} {\bf 115} (2005), 65--78.

\smallskip





\bibitem{lusztig}
{\sc G.~Lusztig}, 
\emph{Hecke algebras with unequal parameters}, CRM Monograph Series 
{\bf 18}, American Mathematical Society, Providence, RI (2003), 136 pp.

\smallskip



\bibitem{pietraho} {\sc T. Pietraho}, 
Cells and constructible representations in type $B$, 
{\it New York J. Math.} {\bf 14} (2008), 411--430.


\end{thebibliography}
\end{document}